\documentclass[article]{siamonline190516}
\usepackage{amsmath} 
\usepackage{verbatim}
\usepackage{textcomp}
\usepackage{amsfonts}
\usepackage{graphicx}
\usepackage{ragged2e}
\usepackage{mathrsfs}
\usepackage{ulem}
\usepackage{xcolor}
\usepackage[margin=1in]{geometry}

\newcommand{\z}{\mathbb{Z}}
\newcommand{\R}{\mathbb{R}}

\newcommand{\C}{\mathbb{C}}
\newcommand{\K}{\mathcal{K}}
\newcommand{\E}{\mathbb{E}}
\newcommand{\bpf}{\begin{proof}}
	\newcommand{\epf}{\end{proof}}

\newtheorem{remark}{Remark}

\title{Robust Approximation of the Stochastic Koopman Operator \thanks{Submitted to the editor April 21, 2021}}

\author{Mathias Wanner\thanks{Department of Mechanical Engineering, 
		University of California, Santa Barbara (\email{mwanner@ucsb.edu},\email{mezic@ucsb.edu})}
	\and
	Dr. Igor Mezi\'{c} \footnotemark[2]
}

\begin{document}
	\maketitle
	\begin{abstract}
		We analyze the performance of Dynamic Mode Decomposition (DMD)-based approximations of the stochastic Koopman operator for random dynamical systems where either the dynamics or observables are affected by noise. For many DMD algorithms, the presence of noise can introduce a bias in the DMD operator, leading to poor approximations of the dynamics. In particular, methods using time delayed observables, such as Hankel DMD, are biased when the dynamics are random. We introduce a new, robust DMD algorithm that can approximate the stochastic Koopman operator despite the presence of noise. We then demonstrate how this algorithm can be applied to time delayed observables, which allows us to generate a Krylov subspace from a single observable. This allows us to compute a realization of the stochastic Koopman operator using a single observable measured over a single trajectory. We test the performance of the algorithms over several examples.
	\end{abstract}
	
	

	%\textbf{Highlights:}
	%\begin{itemize}
	%	\item Developed a new, robust Dynamic Mode Decomposition algorithm that accounts for both the randomness in the dynamics and noise on the observables.
	%	\item Adapted the algorithm to provide a Krylov subspace based method which only requires the data from a single observable.
	%	\item Verified the new algorithms on numerical examples.
	%\end{itemize}
	
	%\newpage
	%	\textbf{Declaration of Interests:}\\
	%	
	%	The authors declare that they have no known competing financial interests or personal relationships that could have appeared to influence the work reported in this paper.\\

	\section{Introduction}
	For many complex systems and processes, governing equations cannot be derived through first principles or the models generated by them may be too complicated to be of practical use. Additionally, for such a system, the true state of the system may be difficult or even impossible to measure, making a state-space model impractical for applications such as control or prediction. Instead, only a limited set of measurements, or observables, will be made available. One of the tools available to model such a system is the Koopman operator. The Koopman operator represents a system in a high-dimensional linear space, which allows us to use spectral methods to analyze the system.
	
	Originally introduced in \cite{Koopman}, the Koopman Operator has gained traction for its utility as a data driven method through various form of Koopman Mode Decomposition (KMD), which decomposes the system based on eigenfunctions of the Koopman operator \cite{Mezic2005}, \cite{Mezic2019}. Introduced in \cite{Mezic2000}, Generalized Laplace Analysis (GLA) is an early data driven method of KMD based on the generalized Laplace transform. Another data driven method is Dynamic Mode Decomposition (DMD), which was introduced in \cite{Schmid} and shown to be connected to KMD in \cite{Rowley}. DMD algorithms attempt to find a matrix which approximates a finite section of the Koopman operator \cite{Mezic2020}. There are many different variations of DMD and it can be used for a wide array of applications. Despite their widespread use, many DMD algorithms possess a major drawback; they can fail if the data contains noise or other randomness.
	
	We can also use the Koopman methodology for random systems, provided the system is a homogeneous Markov process. These include, for example, stochastic differential equations driven by Gaussian white noise and discrete systems generated by generated by i.i.d. random maps. For such systems, the eigenvalues produced by standard DMD algorithms converge to the spectrum of the stochastic Koopman operator, provided the observables themselves do not contain any randomness and lie within a finite dimensional invariant subspace \cite{Williams2015data}. However, if the observables do contain noise, the results from standard DMD algorithms are biased \cite{Dawson}.
	Total Least Squares (TLS) DMD (\cite{Dawson}, \cite{hemati2017biasing}) was developed to remove the bias for systems with measurement noise, but only converges when the underlying dynamics are deterministic. In \cite{Kawahara}, subspace DMD was introduced to converge for observables with additive noise even when the underlying dynamics are random. While many of these methods can combat the bias from measurement noise in DMD, they impose relatively strict assumptions on either the dynamics or the structure of the noise.
	
	Of particular interest are Krylov subspace based DMD methods, where the iterates of a single observable under the Koopman evolution is used to (approximately) generate an invariant subspace of the Koopman operator \cite{brunton2017chaos},\cite{Mezic2020}. For deterministic systems, Hankel DMD uses time delays of a single observable to generate the Krylov subspace, and was shown to converge in \cite{Arbabi}. This allows us to generate a model of a deterministic system using the data from a single trajectory of a single observable. However, for random systems, the time delayed observables contain randomness from the dynamics, and Hankel DMD does not converge. Further, the noise introduced is neither i.i.d. nor independent of the state. In \cite{Nelly}, a new Stochastic Hankel DMD algorithm was shown to converge, but it requires the Stochastic Koopman evolution of the observable, which in general requires multiple realizations of the system. 
	
	In this paper, we introduce a new DMD algorithm which allows us to work with a more general set of observables with noise. This algorithm provably approximates the stochastic Koopman operator in the large data limit and allows for more general  randomness in the observables than i.i.d. measurement noise. With these weaker conditions, we can use time delayed observables to form a Krylov subspace of observables, which gives us a variation of Hankel DMD for random systems. This allows us to compute a realization of the stochastic Koopman operator using data from a single observable over a single realization of the system. The paper is organized as follows: First we review the basics of random dynamical systems and the stochastic Koopman operator. Then, we establish the convergence of standard DMD algorithms for random systems in the absence of noise. Finally, we demonstrate the failure of standard DMD algorithms in the presence of noise and introduce a new algorithm which can accurately approximate the stochastic Koopman operator using noisy observables.

	\section{Preliminaries}
	
	In this paper, we consider random dynamical systems that are generated by random i.i.d. maps. We can use these systems to represent homogeneous Markov processes: systems for which transition probabilities are completely determined by the state; they depend neither on the time nor the past of the system. The random systems considered in this paper will be as follows:

	Let $(\Omega,\mathfrak{F},P)$ be a probability space, and let ($\mathbb{T}=\z$ or $\z^+$) be a semigroup. Let $\{\theta_t\}_{t\in\mathbb{T}}$ be a group or semigroup of measurable transformations on $\Omega$ which preserve the measure $P$. This forms a measure preserving dynamical system on $\Omega$. We will denote the one step map $\theta:=\theta_1$.
	Now, let $(M,\mathfrak{B})$ be a measurable space. Let $T$ be a function that associates to each $\omega$ a random map $T_\omega$ such that:
	\begin{enumerate}
		\item The map $T_\omega:M\to M$ is $\mathfrak{B}$ measurable, and
		\item The maps $T_{\theta_t\omega}$, $t\in \mathbb{T}$, are independent and identically distributed.
	\end{enumerate}
	If the maps $T_{\omega}$ satisfy the above properties, we will call $T$ an i.i.d. random system. For the duration of the paper, the above two properties will always be assumed to hold.

	Now, given a $T$ that satisfies the above properties, we can define the $n$-step evolution of $T$ by
	\begin{equation}
	\label{eq:MapIter}
	T_\omega^n=T_{\theta_{n-1}\omega}\circ T_{\theta_{n-2}\omega}\circ \hdots \circ T_{\theta\omega}\circ T_\omega.
	\end{equation}
	For a given $x\in M$ and $\omega\in \Omega$, the points $T_\omega^tx$ are the trajectory of $x$. We will denote $x_t=T_{\omega_0}^tx_0$ and $\omega_t=\theta_t\omega$ when considering a sample trajectory with initial conditions $x_0$ and $\omega_0$.
	
	These systems driven by i.i.d. maps can be used to represent homogeneous Markov processes. If $M$ is a polish space, for any discrete time Markov process we can find a set of i.i.d. maps that satisfies the transition probabilities of the Markov process (\cite{Kifer}, Theorem 1.1). In this paper, we only need to consider the discrete time case, since the algorithms considered only require a discrete set of data.
	
	\begin{remark}
	If we let $T_\omega^0$ equal the identity, we can immediately see from (\ref{eq:MapIter}) that the maps $T_\omega^t$ form a cocycle over the positive integers:
	\begin{equation}
	\label{eq:cocycle}
	T_\omega^0=id_M,~~~ \text{and} ~~~~~ T_\omega^{t+s}=T_{\theta_s(\omega)}^t\circ T_\omega^s.
	\end{equation}
	With this in mind, the $T$ is also a random dynamical system on $\z_+$ in the sense of \cite{Arnold1998}. However, these are a more general class of system and do not necessarily have the i.i.d. property we require in this paper.
	\end{remark}

		\subsection{Koopman Operators}
		
		Typically, we will not have access the state of the system at any given time. Instead, we will be able to measure some set of functions on the state space.
		\begin{definition} 
			\label{def:observable}
			An observable is any $\mathfrak{B}$ measurable map $f:M\to \C$.
		\end{definition}
		We are interested in the evolution of observables over time. For a deterministic system, the Koopman family of operators is defined to evolve an observable, $f$, on the state space under the flow, $S^t$ of the system: $U^tf=f\circ S^t$. Studying the Koopman evolution of observables has several benefits. First, the state of the system can be reconstructed from a sufficient set of observables, so no information is lost moving to the observable space. Second, a certain choice of observables may lead to a simple representation of the system with linear dynamics. Additionally, since the Koopman operator is a linear operator, it allows us to use spectral methods to study the system.
		
		However, since a random system can have many possible realizations we cannot simply define the stochastic Koopman operator as the composition with the flow. Instead, the stochastic Koopman operators are defined using the expectation of the evolution of observables.
		\begin{definition} 
			\label{def:StochKoop}
			The stochastic Koopman operator, $\K^t$, is defined for for i.i.d. random systems by
			\[\K^tf(x)=\E_{P}(f\circ T_\omega^t(x))=\int_\Omega f\circ T_\omega^t(x)dP.\]
		\end{definition}
		For a discrete time map we will denote the one step Koopman evolution as $\K^1=\K$. The operators $\K^t$ are also called the transition operators for a Markov process.
		
		In order for definition \ref{def:StochKoop} to be useful, we require this family of operators to be consistent in a certain sense: the stochastic Koopman family of operators should form a semigroup:
		\begin{equation}
		\label{eq:semigroup}
		\K^{t+s}f=\K^s\circ \K^tf, ~~~~~s,t\geq 0.
		\end{equation}
		For the deterministic Koopman operators, this is clearly true provided the system is autonomous and solutions exist and are unique, since in this case the flow forms a semigroup. For the stochastic Koopman family of operators on a i.i.d. system, the semigroup property is guaranteed by the independence of the maps $T_{\theta_t\omega}^s$ and $T_\omega^t$, since
		\[\K^{t+s}f(x)=\E_P(f(T_\omega^{t+s}x))=\E_P(f(T_{\theta_t\omega}^sT_\omega^t x))=\E_P(\K^sf(T_\omega^tx))=\K^t\K^sf(x).\]\\
		
		\subsection{Stationarity and Ergodicity}
		We will also assume that our systems will be stationary, meaning they have a stationary measure, $\mu$.
		\begin{definition}
			\label{def:invariant}
			A measure $\mu$ is called invariant, or stationary, if
			\[\mu(A)=\int_M \int_\Omega\chi_A(T_\omega x) dP\,d\mu,\]
			where $\chi_A$ is the indicator function for $A\subset M$.
		\end{definition}
		If $\mu$ is a stationary measure, we have the equality
		\begin{equation}
		\label{eq:stationary}
		\int_M\int_\Omega f(T_\omega^{t_1+s}x,...,T_\omega^{t_n+s}x)\,dPd\mu=\int_M\int_\Omega f(T_\omega^{t_1}x,...,T_\omega^{t_n}x)\,dPd\mu
		\end{equation}
		for any $s, t_1,...,t_n$ (\cite{Gikhman}, p.86).
		
		Since our DMD algorithms will be using data sampled off of a single trajectory of our system, we need the trajectory to sample the measure $\mu$. For this we require that $\mu$ be an ergodic measure.
		\begin{definition}
			A set $A \subset M$ is called invariant if 
			\[\int_\Omega \chi_A(T_\omega x)=\chi_A(x)\]
			for almost every $x$.
		\end{definition}
		\begin{definition}
			A stationary measure $\mu$ is called ergodic if every invariant set has measure $0$ or $1$.
		\end{definition}
		The ergodicity assumptions ensures that almost every trajectory samples the entire space, not just some invariant subset. With this assumption, we can use time averages to evaluate integrals over the space.
		\begin{lemma}
			\label{le:ergodic}
			Suppose $\mu$ is an ergodic measure. Let
			\[h(x,\omega)=\hat{h}(T_\omega^{t_1}x,T_\omega^{t_2}x,...,T_\omega^{t_n}x)\]
			for some $t_1,t_2,...,t_n$, with $h\in L^1(\mu\times P)$. Then we have
			\begin{equation}
			\label{eq:ergodicsum}
			\lim_{m\to\infty}\frac{1}{m} \sum_{j=0}^{m-1} h(x_j,\omega_j)=\int_M\int_\Omega h(x,\omega)dPd\mu
			\end{equation}
			for almost every $(x_0,\omega_0)$ with respect to $\mu\times P$.
		\end{lemma}
		\begin{proof}
			This is theorem 2.2 in chapter 1 of \cite{Kifer}, applied to the sum on the left hand side of (\ref{eq:ergodicsum}).
		\end{proof}

	\section{Dynamic Mode Decomposition}
	Dynamic Mode Decomposition is an algorithm which allows the computation of an approximation of the Koopman operator from data. Assuming the eigenfunctions, $\phi_j$, of $\K$ span our function space, we can decompose any (possibly vector valued) observable $\mathbf{f}$ as 
	\[\mathbf{f}=\sum_j v_j \phi_j.\]
	The expected evolution of $f$ is then given by
	\begin{equation}
	\label{eq:KoopExpansion}
	\E_P(\mathbf{f}(T_\omega x))=\sum_j v_j \K\phi_j(x)=\sum_j \lambda_j v_j \phi_j(x).
	\end{equation}
	In this Koopman mode decomposition, the functions $\phi_j$ are the Koopman eigenfunctions with eigenvalue $\lambda_j$, and the vectors $v_j$ are called the Koopman modes associated with $\mathbf{f}$. However, the expansion above can contain an infinite number of terms. In order to work with (\ref{eq:KoopExpansion}) using finite arithmetic, we must restrict ourselves to a finite dimensional subspace of our original function space.
	
	Let $\mathscr{F}$ be a finite dimensional subspace of $L^2(\mu)$ and $\bar{\mathscr{F}}$ be its orthogonal complement. Let $P_1$ and $P_2$ be the projections on to $\mathscr{F}$ and $\bar{\mathscr{F}}.$ For any function $g \in L_2(\mu)$, we can compute the Koopman evolution as
	\[\K g= P_1\K g +P_2\K g= P_1\K P_1g+ P_2\K P_1g + P_1\K P_2g+P_2\K P_2g.\]
	The operator $P_1\K P_1$ maps $\mathscr{F}$ into itself. For any $g\in \mathscr{F}$, we have $P_2 g=0$, so we can view $P_1\K P_1$ as an approximation of $\K$ provided $\|P_2\K P_1\|$ is small. If $\mathscr{F}$ is an invariant subspace under $\K$, we have $\|P_2\K P_1\|=0$, and $\K g=P_1 \K P_1 g$ for all $g\in \mathscr{F}$. If we let $f_1,f_2,...,f_k$ be a basis for $\mathscr{F}$, we can represent the restriction of $P_1\K P_1$ to $\mathscr{F}$ as a matrix $\mathbf{K}$ that acts on the basis by
	
	\begin{equation}
	\label{eq:finitesub}
	\mathbf{K}\begin{bmatrix}
	f_1 & f_2 & \hdots & f_k
	\end{bmatrix}^T=
	\begin{bmatrix}
	\K f_1 & \K f_2 & \hdots & \K f_k
	\end{bmatrix}^T.
	\end{equation}

	\begin{remark}
		The matrix $K$ can also be thought of as the matrix acting (on the right) on the vector of coefficients of functions represented in the basis $f_1,\hdots,f_k$: for any function $g\in\mathscr{F}$ we can write \[g=\sum_{j=1}^k a_jf_j=\mathbf{a}\begin{bmatrix} f_1 & \hdots & f_k\end{bmatrix}^T,\]
		and $\mathbf{a}=\begin{bmatrix} a_1 & \hdots & a_k\end{bmatrix}$ is the row vector of coefficients of $g$. Then $(\mathbf{a}\mathbf{K})$ is the row vector of coefficients for $\K g$, since
		\[\K g=\K(\mathbf{a}\begin{bmatrix} f_1 & \hdots & f_k\end{bmatrix}^T)=\mathbf{a}\begin{bmatrix} \K f_1 & \hdots & \K f_k\end{bmatrix}=\mathbf{a}\mathbf{K}\begin{bmatrix} f_1 & \hdots & f_k\end{bmatrix}^T\]
	\end{remark}

	Dynamic mode decomposition algorithms compute an approximation of the matrix $\mathbf{K}$ from data. If we can measure the observables $f_1,f_2,...,f_k$ along a trajectory $x_0,x_1,...,x_n$, we can form the vector valued observable $\mathbf{f}:M\to\R^k$ by
	\[\mathbf{f}=
	\begin{bmatrix} 
	f_1 & f_2 & \hdots & f_k
	\end{bmatrix}^T.\] 
	Each $\mathbf{f}(x_t)$ is called a data snapshot. Given a data matrix whose columns are snapshots of $\mathbf{f}$, \[D=\begin{bmatrix}
	\mathbf{f}(x_0) & \mathbf{f}(x_1) & \hdots & \mathbf{f}(x_n)
	\end{bmatrix},\]
	we can construct an operator $A:\R^k\to \R^k$, called the DMD operator, which (approximately) maps each data snapshot to the next one, i.e. \[A\mathbf{f}(x_i)\approx\mathbf{f}(x_{i+1}).\]
	Standard DMD algorithms (see \cite{Schmid},\cite{Williams2015data},\cite{Mezic2020},\cite{Kutz} and the sources therein) construct a matrix $C$ to minimize the error \[\sum_{i=0}^{n-1}\|C\mathbf{f}(x_i)-\mathbf{f}(x_{i+1})\|_2^2.\]

	\noindent\rule{\textwidth}{0.5pt}
	Algorithm 1: Extended Dynamic Mode Decomposition\\
	\noindent\rule{\textwidth}{0.5pt}
	Let $x_0,x_1,...,x_n$ be a trajectory of our random dynamical system and $\mathbf{f}:M\to \C^k$ be a vector valued observable on our system.\\
	1: Construct the data matrices
	\[X=\begin{bmatrix}
	\mathbf{f}(x_0) & \mathbf{f}(x_1) & \hdots & \mathbf{f}(x_{n-1})
	\end{bmatrix},~~~~~~ 
	Y=\begin{bmatrix}
	\mathbf{f}(x_1) & \mathbf{f}(x_2) & \hdots & \mathbf{f}(x_n)
	\end{bmatrix}.\]
	2: Form the matrix 
	\[C=YX^\dagger,\]
	where $X^\dagger$ is the Moore-Penrose psuedoinverse.\\
	3. Compute the eigenvalues and left and right eigenvectors, $(\lambda_i,w_i, v_i)$ $i=1,2,...,k$, of $C$. Then the dynamic eigenvalues are $\lambda_i$, the dynamic modes are $v_i$, and the numerical eigenfunctions are given by 
	\[\hat{\phi}_i=w_i^TX.\]
	\noindent\rule{\textwidth}{0.5pt}
	
	Let $f_1,f_2,...,f_k$ be the components of $\mathbf{f}$. If we let $\hat{f}_i$ be the $i^{th}$ row of $X$,
	\[\hat{f_i}=\begin{bmatrix}
	f_i(x_0) &f_i(x_1) & \hdots & f_i(x_{n-1})
	\end{bmatrix},\]
	we see that $\hat{f}_i$ represents $f_i$ by evaluating it along a trajectory. With standard DMD, we construct the DMD operator $C$ represented in the basis $\hat{f_1},\hat{f_2},...,\hat{f_k}$. Similarly, the numerical eigenfunctions, $\hat{\phi}_i$ will be approximations of eigenfunctions of the stochastic Koopman operator evaluated along our trajectory. Unfortunately, depending on the choice of basis, this DMD construction may be numerically unstable. This leads to the second algorithm \cite{Schmid}.

	\noindent\rule{\textwidth}{0.5pt}
	Algorithm 2: SVD based EDMD\\
	\noindent\rule{\textwidth}{0.5pt}
	Let $x_0,x_1,...,x_n$ be a trajectory of our random dynamical system and, $f_1,f_2,...,f_l,\,l\geq k$, be a set of $l$ observables on our system.\\
	1: Construct the data matrices
	\[X=\begin{bmatrix}
	\mathbf{f}(0) & \mathbf{f}(1) & \hdots & \mathbf{f}(n-1)
	\end{bmatrix},~~~~~~ 
	Y=\begin{bmatrix}
	\mathbf{f}(1) & \mathbf{f}(2) & \hdots & \mathbf{f}(n)
	\end{bmatrix}.\]
	2: Compute the truncated SVD of $X$ using the first $k$ singular values.
	\[X=W_kS_kV_k^*.\]
	3: Form the matrix 
	\[A=S_k^{-1}W_k^*YV_k.\]
	4. Compute the eigenvalues and left and right eigenvectors, $(\lambda_i,w_i, u_i)$ $i=1,2,...,k$, of $A$. Then the dynamic eigenvalues are $\lambda_i$, the dynamic modes are 
	\[v_i=WSu_i,\]
	and the numerical eigenfunctions are given by 
	\[\hat{\phi}_i=w_i^TV_k^*.\]
	\noindent\rule{\textwidth}{0.5pt}
	
	The benefit of SVD based DMD is that it is more numerically stable. If $X$ has a large condition number, the pseudoinversion of $X$ can introduce large errors to the DMD operator and make Algorithm 1 unstable. To combat this, Algorithm 2 computes the SVD of $X$ and truncates to include only the dominant singular values. Since $S_k$ has a smaller condition number than $X$, the inversion of $S_k$ in Algorithm 2 is more numerically stable than the psuedoinversion of $X$. Algorithm 2 uses singular values and vectors to choose a basis of observables to construct the DMD operator; the matrix $A$ generated is the same as the one produced by Algorithm 1 using the $k-$dimensional observable $\mathbf{f}_{new}=S_k^{-1}W^*\mathbf{f}$.
	
	\section{Convergence of DMD for Random Systems}
	
	The utility of Algorithms 1 and 2 comes from the convergence of the dynamic eigenvalues and numerical eigenfunctions to eigenvalues and eigenfunctions of $\K$.
	
	\begin{proposition}
		\label{pr:conv1}
		Let $T$ be an i.i.d. random system with ergodic measure $\mu$. Let $\mathscr{F}$ be a $k$ dimensional subspace of $L^2(\mu)$ which is invariant under the action of $\K$, and let $f_1,f_2,...,f_k$ span $\mathscr{F}$. Let $\lambda_{j,n}$ be the dynamic eigenvalues and $v_{j,n}$ be the dynamic modes produced by Algorithm 1 using the trajectory $x_0,x_1,...,x_n$. Then, as $n\to \infty,$ the dynamic eigenvalues converge to the eigenvalues of $\K$ restricted to $\mathscr{F}$ for almost every initial condition $(x_0,\omega_0)$ with respect to $(\mu\times P)$. If the eigenvalues of $\K$ are distinct, the numerical eigenfunctions converge to a sampling of the eigenfunctions along the trajectory.
	\end{proposition}
	The proof of Proposition \ref{pr:conv1} is fairly standard in the DMD literature (e.g. \cite{Williams2015data}) and does not differ from the deterministic case, but we include it for completeness.
	
	\begin{proof}
		Let $f_1,f_2,...,f_k$, and $\mathbf{K}$ be as described in (\ref{eq:finitesub}). Let $X_n$, $Y_n$, and $C_n$ be the matrices produced by Algorithm 1 for the trajectory $x_0,x_1,...,x_n$, and let $\omega_0,\omega_1,...,\omega_n$ be the evolution of the noise.
		Let $\mathbf{f}=
		\begin{bmatrix} f_1 & f_2 &\hdots& f_k\end{bmatrix}^T$ as above. Define the matrices 
		
		\[G_0=\int_M
		\begin{bmatrix}
		f_1 & f_2 & \hdots & f_k
		\end{bmatrix}^T
		\begin{bmatrix}
		f_1^* & f_2^* & \hdots & f_k^*
		\end{bmatrix}d\mu=\int_M \mathbf{f}\,\mathbf{f}^*\,d\mu\]
		and
		\[G_1=\int_M
		\begin{bmatrix}
		\K f_1 & \K f_2 & \hdots & \K f_k
		\end{bmatrix}^T
		\begin{bmatrix}
		f_1^* & f_2^* & \hdots & f_k^*
		\end{bmatrix}d\mu=\int_M \mathbf{K}\,\mathbf{f}\,\mathbf{f}^*\, d\mu =\mathbf{K}G_0.\]
		We can see that $G_0$ has full rank, since if $\mathbf{v}$ was in its nullspace we would have 
		\[\|\mathbf{f}^*\mathbf{v}\|^2=\mathbf{v}^*G_0\mathbf{v}=0,\]
		which implies $\mathbf{v}=0$ since $f_1,f_2,...,f_k$ are linearly independent. This gives us $\mathbf{K}=G_0^{-1}G_1$.\\
		
		Now, let $G_{0,n}=\frac{1}{n}X_nX_n^*$ and $G_{1,n}=\frac{1}{n}X_nY_n^*$. 
		We have $G_{0,n}\to G_0$ and $G_{1,n}\to G_1$ for almost every initial condition $(x_0,\omega_0)$. To see this, by Lemma \ref{le:ergodic} we have
		
		\begin{align*}
		\lim_{n\to\infty}G_{1,n}&=\lim_{n\to\infty} \frac{1}{n}\sum_{m=0}^{n-1} \mathbf{f}(x_{m+1})\mathbf{f}^*(x_{m})=\lim_{n\to\infty} \frac{1}{n}\sum_{m=0}^{n-1} \mathbf{f}(T_{\omega_m} x_{m})\mathbf{f}^*( x_{m})\\
		&=\int_M\int_P \mathbf{f}(T_\omega x)\mathbf{f}^*(x)\,dPd\mu=\int_M \mathbf{K}\,\mathbf{f}(x)\,\mathbf{f}^*(x)\,d\mu=G_1,
		\end{align*}
		and similarly for $G_0$, we have
		
		\[\lim_{n\to\infty}G_{0,n}=\lim_{n\to\infty} \frac{1}{n}\sum_{m=0}^{n-1} \mathbf{f}(x_{m})\mathbf{f}^*(x_{m})=\int_M \int_\Omega \mathbf{f}(x)\mathbf{f}^*(x)
		\,dPd\mu=G_0.\]
		
		Since $G_0$ has full rank and $G_{0,n}\to G_0$, $G_{0,n}$ is full rank for $n$ large enough, so $G_{0,n}^{-1}$ exists and 
		\[\lim_{n\to\infty}G_{0,n}^{-1}G_{1,n}=G_0^{-1}G_1=\mathbf{K}.\]
		Because $G_{0,n}=\frac{1}{n}X_nX_n^*$, we know $X_n$ has full row rank for $n$ large enough, so \[C_n=Y_n(X_n)^\dagger =Y_nX_n^*(X_nX_n^*)^{-1}=\left(\frac{1}{n}Y_nX_n^*\right)\left(\frac{1}{n}X_nX_n^*\right)^{-1}=G_{0,n}^{-1}G_{1,n},\]
		which shows that $C_n\to\mathbf{K}$. This shows that the dynamic eigenvalues, $\lambda_{j,n}$, converge to the eigenvalues of $\mathbf{K}$, $\lambda_j$, as $n\to \infty$. 
		
		To show the numerical eigenfunctions converge to samplings of our eigenfunctions, let $w_{j,n}$ and $w_j$ be the left eigenvectors of $C_n$ and $\mathbf{K}$, respectively. Consider the functions $\phi_{j,n}=w_{j,n}^T\mathbf{f}$ and $\phi_j=w_j^T\mathbf{f}$. We know $\phi_j$ is a Koopman eigenfunction, since
		\[\K\phi_j=\K(w_j^T\mathbf{f})=w_j^T\mathbf{K}\,\mathbf{f}=\lambda_jw_j^T\mathbf{f}=\lambda_j\phi_j.\]
		If $\mathbf{K}$ has distinct eigenvalues, the vectors $w_{j,n}$ each converge to $w_j$, so $\phi_{j,n}\to\phi_j$. The numerical eigenfunctions, $\hat{\phi}_{j,n},$ are the values of the function $\phi_{j,n}$ sampled along the trajectory $x_0,...,x_{n-1}$.
	\end{proof}
	
	The convergence of Proposition \ref{pr:conv1} is based on the convergence of time averages to inner products of functions in $L^2(\mu)$. In particular, the $i,j^{th}$ entry of $G_{0,n}$ and $G_{1,n}$ converge to $\langle f_i,f_j\rangle$ and $\langle \K f_i,f_j\rangle$, respectively, where $\langle\cdot,\cdot\rangle$ is the $L^2(\mu)$ inner product. As such, we cannot glean any information about dynamics outside the support of $\mu$. There could be an eigenvalue/eigenfunction pair, $(\lambda,\phi)$, such that $\phi$ is zero on the support of $\mu$. Such a pair cannot be captured by Algorithm 1, since $\phi=0$ almost everywhere with respect to $\mu$. In particular, if $\mu$ is a singular measure concentrated on some attractor, the eigenvalues governing the dissipation to the attractor cannot be found using ergodic sampling. 
	
	\section{DMD with Noisy Observables}
	\subsection{Preliminaries}
	The proof above shows that Dynamic Mode Decomposition converges for random dynamical systems with i.i.d. dynamics. However, it is important to note that although the systems can have randomness, the observables cannot. The stochastic Koopman operator acts on functions, $f:M\to \C$, which depend only on the state of the system. If we allow our observables to have some noise (i.e. dependence on $\omega$), the proof fails. In particular, observables with i.i.d. measurement noise and time delayed observables (used in Hankel DMD) both have some dependence on $\omega$, and therefore cannot be used with the above DMD methods. 
	
	Examining the failure of standard DMD with noisy observables is instructive. First we must define our requirements for ``noisy observables.''	
	\begin{definition}
		\label{def:noisyobs}
		A noisy observable is a measurable map $\tilde{f}:M\times\Omega \to \C$, such that the random function $\tilde{f}_\omega=\tilde{f}(\,\cdot\,,\omega):M\to \C$ is $\mathfrak{B}$ measurable for almost every $\omega$.
	\end{definition} 
	For notation, we will always denote a noisy observable, $\tilde{f}$, with a tilde and let the space of noisy observables be $\mathscr{H}$. We will also define $f$ to be its mean:
	\[f(x)=\int_\Omega \tilde{f}_\omega(x) dP.\]
	
	With these definitions, we can interpret $f$ as the ``true'' observable on the system, whereas $\tilde{f}$ is the ``measured'' observable, which comes with some degree of uncertainty. We are interested in the evolution of $f$ rather than $\tilde{f}$, since it depends only on the evolution in the state space and not the noise. Computing the DMD operator with the evolution of $\tilde{f}$ can fit the model to the noise and give a poor approximation of the system.
	In what follows, we will assume that $f$ exists and is in $L^2(\mu)$. To avoid some clutter in the equations and algorithms, we will also denote the time samples of an observable with a hat: $\hat{f}(t)=\tilde{f}(x_t,\omega_t)$.

	In order to evaluate the stochastic Koopman evolution of $f$, we will need to place further restrictions on $\tilde{f}$. We need the random function $\tilde{f}_{\omega}$ to be independent from the past of the dynamics. Precisely, we require that $\tilde{f}_{\theta_t\omega}$ is independent of $T_{\theta_t\omega}$ for all $s<t$ (which implies $\tilde{f}_{\omega_t}$ is independent of $T_{\omega_s}$ for all $s<t$ for any sample path). Roughly speaking, this means the random function $\tilde{f}_{\omega_t}$ cannot be predicted by the past of the dynamics on $M$. The independence condition gives us
	
	\begin{equation}
	\label{eq:NoisyAve}
	\int_\Omega \tilde{f}_{\theta_j\omega}(T_{\omega}^jx)dP(\omega)=\int_\Omega \int_\Omega \tilde{f}_\psi(T_{\omega}^jx)dP(\psi)dP(\omega)=\int_\Omega f(T_{\omega}^jx)dP=\K^jf(x).
	\end{equation}
	
	\begin{comment}
	Similarly, if we have a process $\tilde{g}$ such that $\tilde{g}_{\omega_t}$ is independent of $T_{\omega_s}$ for $s\geq t$, we will say that $\tilde{g}\in\mathscr{H}_-$. These are the observable that cannot be used to predict the future of the dynamics on $M$. Similar to (\ref{eq:NoisyAve}), we have
	\begin{equation}
	\label{eq:NoisyAveAdj}
	\E(\tilde{g}_{\omega}(x)|T_\omega^tx)=\E(g(x)|T_\omega^tx)=\E_{M^\mathbb{T}}(g(u_{s-t})|u_s=x)=\K^*g(x)
	\end{equation}
	\end{comment}
	
	Finally, in order to approximate integrals from data, we will need some ergodicity assumptions on our noisy observables. Namely, we will need time averages to converge in a similar sense to Lemma \ref{le:ergodic}. In particular, we will need
	\begin{equation}
	\label{eq:noisyerg}
	\lim_{n\to\infty}\frac{1}{n} \sum_{j=0}^{n-1} \hat{\mathbf{f}}(t+j)\hat{\mathbf{g}}(t)=\int_M\int_\Omega \tilde{\mathbf{f}}_{\theta_j\omega}(T_{\omega}^jx)\tilde{\mathbf{g}}_\omega(x)dP(\omega)d\mu(x),
	\end{equation}
	for two vector valued noisy observables $\tilde{\mathbf{f}}$ and $\tilde{\mathbf{g}}$ and almost every initial condition $(x_0,\omega_0)$.

	\begin{remark}
		While we make the ergodicity assumption for generality, we will show that (\ref{eq:noisyerg}) holds for observables with i.i.d. measurement noise and time delayed observables, the primary observables of interest in this paper. More generally, we can consider the skew product system $\Theta$ on $M\times \Omega$ given by $\Theta(x,\omega)=(T_\omega x, \theta\omega)$ and treat $\tilde{f}$ as an observable on $M\times \Omega$. If $\mu\times P$ is an ergodic measure for $\Theta$, we can evaluate time averages as in (\ref{eq:noisyerg}).
	\end{remark}

	\subsection{Failure of Dynamic Mode Decomposition with Noisy Observables}
	Now, using the ergodicity (\ref{eq:noisyerg}) and independence (\ref{eq:NoisyAve}) assumptions above, we can see exactly how DMD fails. The convergence of DMD depends largely on estimation of inner products using time averages.	
	As before let $f_1,...,f_k$ be observables which span a $k$-dimensional subspace $\mathscr{F}$, and let $\mathbf{K}$ be the restriction of $\K$ to $\mathscr{F}$ as in (\ref{eq:finitesub}). Let $\mathbf{f}=\begin{bmatrix}
	f_1 & \hdots & f_k
	\end{bmatrix}^T$. We have from Lemma \ref{le:ergodic} that 
	\[G_j=\lim_{n\to\infty}\frac{1}{n}\sum_{m=0}^{n-1} \mathbf{f}(x_{m+j})\mathbf{f}^*(x_{m})=\int_M\int_\Omega \mathbf{f}(T_\omega^jx)\mathbf{f}^*(x)dPd\mu=\int_M \mathbf{K}^j\mathbf{f}\,\mathbf{f}^*\,d\mu.\]
	We can use the fact that $G_j=\mathbf{K}G_{j-1}$ to estimate $\mathbf{K}$.
	However, suppose we have a noisy observable $\tilde{f}\in \mathscr{H}$ with $\E_P(\tilde{\mathbf{f}}_\omega)=\mathbf{f}$ such that (\ref{eq:noisyerg}) and (\ref{eq:NoisyAve}) hold. When we take the comparable time average, we have
	\[\tilde{G}_j=\lim_{n\to\infty}\frac{1}{n}\sum_{m=0}^{n-1}\hat{\mathbf{f}}(m+j)\hat{\mathbf{f}}(m)^*=\int_M \int_\Omega \tilde{\mathbf{f}}_{\theta_j\omega}(T_\omega^jx)\tilde{\mathbf{f}}^*_{\omega}(x)dPd\mu.\]
	This is not equal to $G_j$, since $\tilde{\mathbf{f}}_\omega$ and $\tilde{\mathbf{f}}_{\theta_j\omega}\circ T_\omega^j$ are not necessarily independent. In fact, if we examine the difference in $\tilde{G}_j$ and $G_j$, we obtain
	\[\tilde{G}_j-G_j
	=\int_M Cov(\tilde{\mathbf{f}}_{\theta_j\omega}\circ T_\omega^j,\tilde{\mathbf{f}}_\omega) d\mu,\]
	where $Cov(\tilde{\mathbf{f}}_{\theta_j\omega}\circ T_\omega^j,\tilde{\mathbf{f}}_\omega)(x)$ denotes the covariance of $\tilde{\mathbf{f}}_{\theta_j\omega}(T_\omega^tx)$ and $\tilde{\mathbf{f}}_\omega(x)$, since, using (\ref{eq:NoisyAve}),
	
	\[Cov(\tilde{\mathbf{f}}_{\theta_j\omega}\circ T_\omega^j,\tilde{\mathbf{f}}_\omega)(x)=\E_P(\tilde{\mathbf{f}}_{\theta_j \omega}(T_\omega^j x)\tilde{\mathbf{f}}^*_\omega(x))-\E_P(\tilde{\mathbf{f}}_{\theta_j \omega}(T_\omega^j x))\E_P(\tilde{\mathbf{f}}^*(x))\]
	\[=\E_P(\tilde{\mathbf{f}}_{\theta_j \omega}(T_\omega^j x)\tilde{\mathbf{f}}_\omega(x)^*)-\mathbf{K}^j\mathbf{f}\,\mathbf{f}^*.\]

	Since Algorithms 1 and 2 depend on the numerical approximations of $G_j$, we can conclude that the error stems from the covariances of the observables. However, if we could somehow guarantee this covariance was zero, we could still compute $\mathbf{K}$. We will do this by choosing a second set of observables, $\tilde{\mathbf{g}}_\omega$, such that $Cov(\tilde{\mathbf{f}}_{\theta_j\omega}\circ T_\omega^j,\tilde{\mathbf{g}}_\omega)=0$. We can guarantee this by ensuring $\tilde{\mathbf{g}}$ meets some independence conditions with $T_\omega$ and $\tilde{\mathbf{f}}$. This brings us to our third algorithm, which gives us the freedom to choose $\tilde{\mathbf{g}}$.

	\subsection{Noise Resistant DMD Algorithms}$~$\\
	\noindent\rule{\textwidth}{0.5pt}
	Algorithm 3: Noise Resistant DMD\\
	\noindent\rule{\textwidth}{0.5pt}
	Let $\tilde{\mathbf{f}}\in \mathscr{H}^k$, and $\tilde{\mathbf{g}}\in \mathscr{H}^l,~l\geq k$. As before, let $\hat{\mathbf{f}}(t)=\tilde{\mathbf{f}}(x_t,\omega_t)$ and $\hat{\mathbf{g}}(t)=\tilde{\mathbf{g}}(x_t,\omega_t)$ denote their samples along a trajectory at time $t$.\\
	1: Construct the data matrices	
	\[X=\begin{bmatrix}
	\hat{\mathbf{f}}(0) & \hat{\mathbf{f}}(1) & \hdots & \hat{\mathbf{f}}(n-1)
	\end{bmatrix},\]
	\[Y=\begin{bmatrix}
	\hat{\mathbf{f}}(1) & \hat{\mathbf{f}}(2) & \hdots & \hat{\mathbf{f}}(n)
	\end{bmatrix},\]
	and
	\[Z=\begin{bmatrix}
	\hat{\mathbf{g}}(0) & \hat{\mathbf{g}}(1) & \hdots & \hat{\mathbf{g}}(n-1)
	\end{bmatrix}.\]
	2: Form the matrices $\tilde{G}_{0}=\frac{1}{n}XZ^*$ and $\tilde{G}_{1}=\frac{1}{n}YZ^*$.\\
	3: Compute the matrix
	\[C= \tilde{G}_{1}\tilde{G}_{0}^\dagger.\]
	4: Compute the eigenvalues and left and right eigenvectors, $(\lambda_i,w_i,v_i)$ of $C$. The dynamic eigenvalues are $\lambda_i$, the dynamic modes are $v_i$, and the numerical eigenfunctions are given by
	\[\hat{\phi}_i=w_i^T X.\]
	\noindent\rule{\textwidth}{0.5pt}
	
	The idea behind Algorithm 3 is to use a second noisy observable, $\tilde{\mathbf{g}}$, which meets some independence requirements with $\tilde{\mathbf{f}}$, to generate a second basis for $\mathscr{F}$. If $\tilde{\mathbf{g}}$ meets the proper independence requirements, the convergence can be shown in a similar manner to Proposition \ref{pr:conv1}.
	
	\begin{proposition}
		\label{pr:conv2}
		Let $\tilde{\mathbf{f}}\in\mathscr{H}^k$ and $\tilde{\mathbf{g}}\in\mathscr{H}^l$ be such that $\tilde{\mathbf{f}}$ and $\tilde{\mathbf{g}}$ satisfy (\ref{eq:noisyerg}) and $\tilde{\mathbf{f}}$ satisfies (\ref{eq:NoisyAve}). Suppose $\tilde{\mathbf{g}}_{\omega}$ is independent of $\tilde{\mathbf{f}}_{\omega}$,  $\tilde{\mathbf{f}}_{\theta\omega}$, and $T_{\omega}$. Define $\mathbf{f}(x)=\E_\Omega(\tilde{\mathbf{f}}_\omega(x))$ and $\mathbf{g}(x)=\E_\Omega(\tilde{\mathbf{g}}_\omega(x)).$ Suppose the components of $\mathbf{f}$, $f_1,...,f_k$, span a $k$-dimensional invariant subspace, $\mathscr{F}$, of $\K$ and $\mathscr{F}\subset \text{span}\{g_1,...,g_l\},$ where $g_1,...,g_l$ are the components of $\mathbf{g}$. Then the matrix $C$ generated by Algorithm 3 converges to the restriction of $\K$ to $\mathscr{F}$ as $n\to \infty$.
	\end{proposition}
	
	\begin{proof}
		Let $\mathbf{K}$ be the restriction of $\K$ to $\mathscr{F}$. Let $\tilde{G}_{0,n}$ and $\tilde{G}_{1,n}$ be the matrices generated in Algorithm 3 with $n$ data points. Using the independence conditions on $\tilde{\mathbf{g}}$, $\tilde{\mathbf{f}}$, and $T_\omega$ and (\ref{eq:NoisyAve}), define
		\begin{equation}
		\label{eq:TimeShiftAve}
		G_0=\int_M\int_\Omega\tilde{\mathbf{f}}_{\omega}(x)\tilde{\mathbf{g}}^*_{\omega}(x)dPd\mu=\int_M\int_\Omega\tilde{\mathbf{f}}_{\omega}(x)dP\int_\Omega\tilde{\mathbf{g}}^*_{\omega}(x)dPd\mu=\int_M \mathbf{f}\,\mathbf{g}^* d\mu
		\end{equation}
		and
		\begin{equation}
		\label{eq:TimeShiftAve2}
		G_1=\int_M\int_\Omega \tilde{\mathbf{f}}_{\theta\omega}(T_{\omega}x)\tilde{\mathbf{g}}^*_{\omega}(x)dPd\mu= \int_M\int_\Omega\tilde{\mathbf{f}}(T_{\omega}x)dP\int_\Omega\tilde{\mathbf{g}}^*_{\omega}(x)dPd\mu=\mathbf{K}\int_M \mathbf{f}\,\mathbf{g}^*d\mu.
		\end{equation}
		We can show that $G_0$ has full row rank; if $\mathbf{v}$ is in its left nullspace, we would have
		\[\langle \mathbf{v}^T\mathbf{f},g_{i}\rangle=0\]
		for each $i$, which shows $\mathbf{v}=0$ since $\mathscr{F}\subset \text{span}\{g_i\}.$ This gives us $\mathbf{K}= G_{1}G_0^\dagger$. We will show that $\tilde{G}_{0,n}\to G_0$ and $\tilde{G}_{1,n}\to G_{1}$ as $n\to \infty.$ Taking the limit of $G_{0,n}$ with (\ref{eq:noisyerg}) and using (\ref{eq:TimeShiftAve}), we have
		\[\lim_{n\to\infty}\tilde{G}_{0,n}=\lim_{n\to\infty}\frac{1}{n}\sum_{m=0}^{n-1}\hat{\mathbf{f}}(m)\hat{\mathbf{g}}^*(m)=\int_M\int_\Omega \tilde{\mathbf{f}}_{\omega}(x)\tilde{\mathbf{g}}^*_{\omega}(x)\,dPd\mu=G_0\]
		and similarly $\tilde{G}_{1,n}\to G_{1}$ using (\ref{eq:TimeShiftAve2}).
		Since $G_0$ has full rank and $\tilde{G}_{0,n}\to G_0$, we have $\tilde{G}_{0,n}^\dagger\to G_0^\dagger$, so $\tilde{G}_{1,n}\tilde{G}_{0,n}^\dagger\to \mathbf{K}$.
	\end{proof}
	
	It follows from Proposition \ref{pr:conv2} that the eigenvalues and eigenvectors of $C$ go to those of $\mathbf{K}$. Therefore, the dynamic eigenvalues limit to Koopman eigenvalues. The numerical eigenfunctions, however, are more complicated. If $w_i$ is a left eigenvector of $\mathbf{K}$, we have $w_i^T\mathbf{f}$ is a Koopman eigenfunction. The numerical eigenfunctions, however, limit to $w_i^TX$, which a sampling of $w_i^T\tilde{\mathbf{f}}$. In this regard, the numerical eigenfunction is a sampling of an eigenfunction with some zero mean noise added to it.
	
	The key idea in the proof of Proposition \ref{pr:conv2} is the assumption that we have a second observable $\tilde{\mathbf{g}}$ that is uncorrelated with $\tilde{\mathbf{f}}$. This allows us to estimate the inner product of $\mathbf{g}$ and $\mathbf{f}$ using time averages without introducing a covariance term. We call $\tilde{\mathbf{g}}$ our ``dual observable'' since we are using it to evaluate these inner products. While the necessity of a second observable may seem restrictive, Proposition \ref{pr:conv2} allows us to work with very general observables. If we specialize to more specific classes of observables, we will find that we often do not need a second observable. Often, we can use time delays of single observable $\tilde{\mathbf{f}}$ so that $\tilde{\mathbf{f}}_{\omega}$ and $\tilde{\mathbf{f}}_{\theta^s\omega}$ are independent.
	
	\subsection{Observables with i.i.d. Measurement Noise}
	Often, when measuring an observable on a system, the measurement will be imprecise. The error in the measurement are often modeled as an i.i.d. random variable. We call an observable with this type of noise an observable with measurement noise:
	\begin{definition}
		A noisy observable, $\tilde{f}$, is an observable with i.i.d. measurement noise if $\tilde{f}_{\theta_t}\omega$ is an i.i.d. random function and is independent of the random maps $T_{\theta_s\omega}$ for all $s$.
	\end{definition}
	\noindent Let $f=\E_P(\tilde{f}_\omega)$. We note that for any given $\omega$, the measurement error,
	\[\tilde{e}_\omega(x)=\tilde{f}_{\omega}(x)-f(x),\]
	can vary over the state space $M$; it does not need to be a constant additive noise. Since $\tilde{f}_{\omega_t}$ is an i.i.d. random variable and independent of $T_{\omega_t}$ for all $t$, the ordered pair $(x_t,\tilde{f}_{\omega_t})\in M\times L^2(M)$ is an ergodic process, with ergodic measure $\nu=\mu\times\tilde{f}_*(P)$, where $\tilde{f}_*(P)$ is the pushforward of $P$. This allows us to evaluate the time averages as in (\ref{eq:noisyerg}). The proof of this follows from the lemma below and the fact that i.i.d. processes are mixing (\cite{Gikhman}, Theorem 4, page 143).

	\begin{lemma}
		\label{le:iiderg}
		Let $x_t$ and $y_t$ be independent stationary processes. If $x_t$ is ergodic and $y_t$ is mixing, then $(x_t,y_t)$ is ergodic.
	\end{lemma}
	\begin{proof}
		The result follows from Theorem 6.1 on page 65 of \cite{PetersonErgodic}, where we can represent the processes as a measure preserving shifts on the space of sequences of $x_t$ and $y_t$ (\cite{PetersonErgodic}, page 6). 
	\end{proof}

	If the components of $\tilde{\mathbf{f}}$ are observables with measurement noise, it turns out we don't need second observable to use in Algorithm 3. Instead, we can use a time shift of $\tilde{\mathbf{f}}$ to generate $\tilde{\mathbf{g}}$. The i.i.d. property of $\tilde{\mathbf{f}}$ will give us the independence properties we need.
	
	\begin{corollary}
		\label{cor:measnoise}
		Suppose $\tilde{\mathbf{f}}$ is a vector valued observable with i.i.d. measurement noise, and the components of $\mathbf{f}=\E_P(\tilde{\mathbf{f}}_\omega)$ span a $k$-dimensional invariant subspace, $\mathscr{F}$. Suppose further that the restriction of $\K$ to $\mathscr{F}$ has full rank. Then Algorithm 3 converges setting $\hat{\mathbf{g}}(t)=\hat{\mathbf{f}}(t-1)$.
	\end{corollary} 
	
	\begin{proof}
		Let $\mathbf{K}$ be the resriction of $\K$ to $\mathscr{F}$. By Lemma \ref{le:iiderg}, $(x_t,f_{\omega_t})$ is an ergodic stationary sequence. Then, using ergodicity and the independence properties of $\tilde{\mathbf{f}}$, we have
		\begin{align*}
		\lim_{n\to\infty}\frac{1}{n}\sum_{m=1}^{n}\hat{\mathbf{f}}(m)\hat{\mathbf{g}}^*(m)&=\lim_{n\to\infty}\frac{1}{n}\sum_{m=0}^{n-1}\hat{\mathbf{f}}(m)\hat{\mathbf{f}}^*(m-1)=\int_M\int_\Omega \tilde{\mathbf{f}}_{\theta\omega}(T_\omega^{j+1} x)\tilde{\mathbf{f}}^*_{\omega}(x)\,dPd\mu\\
		&=\int_M \int_\Omega \tilde{\mathbf{f}}_{\theta\omega}(T_\omega x)dP\int_\Omega \tilde{\mathbf{f}}_\omega(x)dPd\mu=\mathbf{K}\int_M \mathbf{f}\,\mathbf{f}^*d\mu,
		\end{align*}
		which has full rank since $\mathbf{K}$ has full rank. Similarly, 
		\[\lim_{n\to\infty}\frac{1}{n}\sum_{m=1}^{n}\hat{\mathbf{f}}(m+1)\hat{\mathbf{g}}^*(m)=\lim_{n\to\infty}\frac{1}{n}\sum_{m=0}^{n-1}\hat{\mathbf{f}}(m+1)\hat{\mathbf{f}}^*(m-1)=\mathbf{K}^2\int_M \mathbf{f}\,\mathbf{f}^*d\mu.\]
		The rest of the proof follows Proposition \ref{pr:conv2}.
	\end{proof}
	
	\begin{remark}
		It is useful to note that if $T_\omega$ and $\theta$ were invertible, we would be able to define $\tilde{\mathbf{g}}_\omega=\tilde{\mathbf{f}}_{\theta_{-1}\omega}\circ(T_\omega^{-1})$, and $\tilde{\mathbf{g}}$ would meet the conditions of Proposition \ref{pr:conv2} exactly. However, if they are not invertible, we cannot necessarily define $\tilde{\mathbf{g}}_\omega\in L^2(M)$ explicitly since $T_\omega$ may not be invertible. However, since we are still able to evaluate time averages, the proof is nearly identical.
	\end{remark}
	
	\section{Time Delayed Observables and Krylov Subspace Methods}
	Another important type of noisy observable are time delayed observables. Allowing time delayed observables in DMD is useful for two reasons. First, time delays allow us to enrich our space of observables. Oftentimes, there are functions on our state space which cannot be measured by a certain set of observables, but can be observed if we allow time delays. For example, the velocity of a moving mass cannot be observed by any function on the position, but can be approximated using the position at two different times. Second, using time delays allows us to identify an invariant (or nearly invariant) subspace spanned by the Krylov sequence $f,\K f,...,\K^{k-1}f$.
	
	Of particular interest is an analogue of Hankel DMD for random systems, which uses a Krylov sequence of observables to generate our finite subspace. With Hankel DMD, we use a single observable, $f$, and its time delays to approximate the sequence $f,\K f,...,\K^{k-1}f$. If $\tilde{f}$ is an observable with measurement noise (or has no noise), we can define
	\[\tilde{\mathbf{f}}(x,\omega)=\begin{bmatrix}
	\tilde{f}(x,\omega) & \tilde{f}(T_\omega x,\theta\omega) & \hdots & \tilde{f}(T_\omega^{k-1}x,\theta_{k-1}\omega)
	\end{bmatrix}^T.\]
	By (\ref{eq:NoisyAve}), its mean is
	\[\int_\Omega \tilde{\mathbf{f}}\,dP=\begin{bmatrix} f & \K f & \hdots & \K^{k-1}f \end{bmatrix}^T,\]
	where $f=\E_P(\tilde{f})$. We can then use time delays of $\tilde{f}$ to approximate the Krylov sequence $f, \K f, ..., \K^{k-1}f$. Additionally, if we set $\tilde{\mathbf{g}}(t)=\tilde{\mathbf{f}}(t-k)$ in Algorithm 3, we will have the necessary independence conditions, and the time averages will converge as in (\ref{eq:noisyerg}) due to the pair $(x_t,\tilde{f}_{\omega_t})$ being an ergodic stationary variable.
	
	\begin{corollary}
		\label{cor:hankel}
		(Noise Resistant Hankel DMD) Let $\tilde{f}$ be an observable with measurement noise, with time samples $\hat{f}(t)=\tilde{f}(x_t,\omega_t).$ Let its mean, $f$, be such that the Krylov sequence $f,\K f,...,\K^{k-1}f$ spans a $k$-dimensional invariant subspace $\mathscr{F}$ and the restriction of $\K$ to $\mathscr{F}$ has full rank. Let 
		\[\hat{\mathbf{f}}(t)=\begin{bmatrix} \hat{f}(t) & \hat{f}(t+1) & \hdots & \hat{f}(t+k-1) \end{bmatrix}^T,\]
		and 
		\[\hat{\mathbf{g}}(t)=\hat{\mathbf{f}}(t-k)=\begin{bmatrix}
		\hat{f}(t-k) & \hat{f}(t-k+1) & \hdots & \hat{f}(t-1)
		\end{bmatrix}^T.\]
		Then the matrix $A$ generated by Algorithm 3 converges to the restriction of $\K$ to $\mathscr{F}$. If $\tilde{f}$ has no noise (i.e. $\tilde{f}(x,\omega)=f(x)$) we can use
		\[\hat{\mathbf{g}}'(t)=\hat{\mathbf{f}}(t-k+1)=\begin{bmatrix}
		\hat{f}(t-k+1) & \hat{f}(t-k+2) & \hdots & \hat{f}(t)
		\end{bmatrix}^T.\]
	\end{corollary}
	We refer to Corollary \ref{cor:hankel} as a variant of Hankel DMD for random systems since the $X, Y,$ and $Z$ matrices in Algorithm 3 will be Hankel matrices and it generates a Krylov subspace of $\K$. For a different choice of $\tilde{\mathbf{g}}$ (i.e. $\tilde{\mathbf{g}}=\tilde{\mathbf{f}}$), this is equivalent to Hankel DMD.
	
	\begin{proof}
		Using (\ref{eq:NoisyAve}), we can see that the components of $\mathbf{f}$ are $f,\K f,...,\K^{k-1}f$, which spans $\mathscr{F}$. Additionally, using the independence properties of $\tilde{f}$, we have $\tilde{\mathbf{f}}_{\omega_t}$ and $\tilde{\mathbf{f}}_{\omega_{t+s}}$ are independent for $s\geq k$. Since $(x_t,\tilde{f}_{\omega_t})$ is ergodic by Lemma \ref{le:iiderg}, we can take the time averages
		\begin{align*}
		\lim_{n\to\infty}\frac{1}{n}\sum_{m=k}^{n+k-1}\mathbf{f}(m)\mathbf{g}^*(m)&=\lim_{n\to\infty}\frac{1}{n}\sum_{m=0}^{n-1}\mathbf{f}(m+k)\mathbf{f}^*(m)=\int_M\int_\Omega \tilde{\mathbf{f}}_{\theta^k\omega}(T_\omega^kx)\tilde{\mathbf{f}}^*_\omega(x)dPd\mu\\
		&=\int_M\int_\Omega \tilde{\mathbf{f}}_{\theta_k\omega}(T_\omega^kx)\tilde{\mathbf{f}}^*(x)dPd\mu=\mathbf{K}^k\int_M \mathbf{f}\,\mathbf{f}^*d\mu,
		\end{align*}
		which has full rank since $\mathbf{K}$ has full rank. Similarly, we can take the time average
		\[\lim_{n\to\infty}\frac{1}{n}\sum_{m=k}^{n+k-1}\mathbf{f}(m+1)\mathbf{g}^*(m)=\mathbf{K}^{k+1}\int_M\mathbf{f}\,\mathbf{f}^*d\mu,\]
		and the rest of the proof follows Proposition \ref{pr:conv2}. If $\tilde{f}_\omega=f$, $\tilde{\mathbf{f}}_{\omega_t}$ and $\tilde{\mathbf{f}}_{\omega_{t+k-1}}$ are independent, and we can take the time averages using $\hat{\mathbf{g}}(t)=\hat{\mathbf{f}}(t-k+1)$.
		
	\end{proof}
	
	\begin{comment}
	\begin{proof}
	Using (\ref{eq:NoisyAve}) and (\ref{eq:NoisyAveAdj}), we can see that the components of $\mathbf{f}$ are $f,\K f,...,\K^{k-1}f$ and the components of $\mathbf{g}$ are $(\K^*)^{k}f,\hdots,\K^*f$, which both span $\mathscr{F}$. Additionally, $\tilde{f}_{\omega_t}$ is independent of $\tilde{f}_{\omega_s}$ and $\tilde{f}_{\omega_{s+1}}\circ T_{\omega_s}$ for $s\neq t$ since it is an observable with measurement noise. We can use this to verify that $\tilde{\mathbf{g}}_{\omega_t}$ is independent of $\tilde{\mathbf{f}}_{\omega_t}$ and $\tilde{\mathbf{f}}_{\omega_{t+1}}\circ T_{\omega_t}$ componentwise and the observables meet the conditions of Proposition \ref{pr:conv2}. If $\tilde{f}$ has no noise, the independence properties on $\tilde{f}$ hold for $s=t$ and we can verify that $\tilde{\mathbf{g}}'$ has the independence properties. Finally, since $(x_t,\tilde{f}_{\omega_t})$ is ergodic, the time averages converge as desired.
	\end{proof}
	\end{comment}
	
	Corollary \ref{cor:hankel} allows us to compute an approximation of $\K$ using the data from a single observable evaluated along a single trajectory. However, the method does not require that the we only use time delays of a single observable. In general, even if $\tilde{f}$ is vector valued, we can take time delays of $\tilde{f}$ as in Corollary \ref{cor:hankel} so long as we span the proper subspace. The dual observable, $\tilde{\mathbf{g}}$, is also generated in the same way.

	\section{Conditioning of Algorithm 3}
	Asymptotically, the convergence rate of Algorithm 3 is governed by the rate at which $G_{0,n}$ and $G_{1,n}$ converges to $G_0$ and $G_1$, as defined in the proof of Proposition \ref{pr:conv2}. This is governed by the convergence rate of ergodic sampling. However, Algorithm 3 also requires the pseudo-inversion of $G_{0,n}\approx G_0$. If the matrix $G_0$ is ill-conditioned, small errors in the time averages approximations of $G_0$ and $G_{1}$ can cause large errors in our DMD operator.	The condition number of $G_0$, $\kappa(G_0)$, can become large if either set of observables, $f_1,...,f_k$ or $g_1,...,g_l$, are close to being linearly dependent.
	
	Both of these issues arise particularly often when using Hankel DMD. With Hankel DMD, we use the basis $f,\K f, ...,\K^{k-1}f$ as our basis for $\mathscr{F}$. This is often a poor choice of basis, as $f$ and $\K f$ may be close to being linearly dependent. This is particularly the case when data from a continuous time system is sampled with a short period, such as from a discretization of an ODE or SDE. Similarly, if $j$ is large or $\K$ has eigenvalues close to zero, $\K^j f$ and $\K^{j+1}f$ may be close to being linearly dependent, which will also cause conditioning issues.
	
	\subsection{SVD Based Algorithms}
	
	To combat these conditioning issues, we have some leeway in the observables we choose for $\tilde{\mathbf{f}}$ and $\tilde{\mathbf{g}}$. Looking at $G_0$, we have
	\begin{equation}
	\label{eq:DaulBasis}
	G_0=\int_M \mathbf{g}\,\mathbf{f}^*\,\,d\mu=\int_M \begin{bmatrix} f_1 & f_2 & \hdots & f_k \end{bmatrix}^T\begin{bmatrix} g_1^* & g_2^* & \hdots & g_l^*\end{bmatrix}
	d\mu.
	\end{equation}
	Ideally, $\{g_1,...,g_l\}$ and $\{f_1,...,f_k\}$ would be orthonormal bases for $\mathscr{F}$, so $\kappa(G_0)$ would be $1$. However, we rarely can choose such bases a priori. Instead, we can try to augment $\tilde{\mathbf{f}}$ and $\tilde{\mathbf{g}}$ with extra observables and use the singular value decomposition to choose $k$ observables which form a better conditioned basis for $\mathscr{F}$, similar to Algorithm 2. This brings us to the SVD implementation of Algorithm 3.

	\noindent\rule{\textwidth}{0.5pt}
	Algorithm 4: SVD implemented Noise Resistant DMD\\
	\noindent\rule{\textwidth}{0.5pt}
	Let $\tilde{\mathbf{f}}\in \mathscr{H}^{l_1}$, and $\tilde{\mathbf{g}}\in \mathscr{H}^{l_2},~l_1,l_2\geq k$ be noisy observables on our system. Let $\hat{\mathbf{f}}(t)=\tilde{\mathbf{f}}(x_t,\omega_t)$ and $\hat{\mathbf{g}}(t)=\tilde{\mathbf{g}}(x_t,\omega_t)$ denote the time samples of the observables.\\
	1: Construct the data matrices	
	\[X=\begin{bmatrix}
	\hat{\mathbf{f}}(0) & \hat{\mathbf{f}}(1) & \hdots & \hat{\mathbf{f}}(n-1)
	\end{bmatrix},\]
	\[Y=\begin{bmatrix}
	\hat{\mathbf{f}}(1) & \hat{\mathbf{f}}(2) & \hdots & \hat{\mathbf{f}}(n)
	\end{bmatrix},\]
	and
	\[Z=\begin{bmatrix}
	\hat{\mathbf{g}}(0) & \hat{\mathbf{g}}(1) & \hdots & \hat{\mathbf{g}}(n-1)
	\end{bmatrix}.\]
	2: Form the matrices $\tilde{G}_{0}=\frac{1}{n}XZ^*$ and $\tilde{G}_{1}=\frac{1}{n}YZ^*$.\\
	3: Compute the truncated SVD of $\tilde{G}_0$ using the first $k$ singular values:
	\[\tilde{G}_0 \approx W_kS_kV_k^*.\]
	5: Form the matrix 
	\[A=S_k^{-1}W_k^*\tilde{G}_{1}V_k.\]	
	6: Compute the eigenvalues and left and right eigenvectors, $(\lambda_i,w_i,u_i)$ of $A$. The dynamic eigenvalues are $\lambda_i$, the dynamic modes are
	\[v_i=W_kS_ku_i,\]
	and the numerical eigenfunctions are
	\[\hat{\phi}_i=w_iS_k^{-1}W_k^*X.\]
	\noindent\rule{\textwidth}{0.5pt}
	
	Similar to Algorithm 2, Algorithm 4 uses the SVD to choose a basis of observables to use in Algorithm 1. It is equivalent to performing Algorithm 3 using data from the observable $(S_k^{-1}W_k^*)\tilde{\mathbf{f}}$, while leaving $\tilde{\mathbf{g}}$ unchanged. It is important to note that Algorithm 4 uses the components of $(S_k^{-1}W_k^*)\mathbf{f}$ to as a basis for $\mathscr{F}$ where $\mathbf{f}=\E_P(\tilde{\mathbf{f}})$ as usual. When we add observables to $\tilde{\mathbf{f}}$, we must ensure that we stay within our invariant subspace. One way to guarantee this is to use time delays of our original observables.
	
	\subsection{Augmented Dual Observables}
	
	Typically, augmenting $\tilde{\mathbf{f}}$ with extra observables and using Algorithm 4 to truncate the singular values is an effective way to improve the conditioning of the problem. However, we have an alternate tool at our disposal. While each component of $\mathbf{f}$ must lie within $\mathscr{F}$, the components of $\mathbf{g}$ can be arbitrary, and we do not need to take an SVD to truncate the extra observables in $\mathbf{g}$. Since we do not need to worry about leaving our invariant subspace, we can add arbitrary functions of $\tilde{\mathbf{g}}$ (e.g. powers of $\tilde{\mathbf{g}}$) to our dual observable and still expect convergence. However, while this can improve conditioning, it also can slow down the convergence of the time averages, and should only be done when the error stems from poor conditioning.

	\section{Numerical Examples}
	In this section, we will test the various DMD algorithms presented in this paper using both observables with measurement noise and time delayed observables. For each system and each DMD method, we generate five realizations of the DMD operator and compare the eigenvalues with analytically obtained true (or approximate) eigenvalues of the stochastic Koopman eigenvalues. Since the purpose of this paper is to provide a new algorithm that is provably unbaised, we only compare the noise resistant algorithms to standard DMD algorithms. Comparisons on the speed of convergence and numerical stability of various DMD algorithms not the primary purpose of this paper.
	\subsection{Random Rotation on a Circle}
	
	Consider a rotation on the circle. The dynamical system is defined by
	\begin{equation}
	\label{eq:Rotation}
	x_{t+1}=x_t+\nu,
	\end{equation}
	where $\nu\in S^1$. If we perturb (\ref{eq:Rotation}) by adding noise to the rotation rate we obtain the random system
	\begin{equation}
	\label{eq:RandomRotation}
	x_t+1=x_{t}+\nu+\pi(\omega_t)
	\end{equation}
	where $\pi(\omega_t)\in S^1$ is an i.i.d. random variable. For the stochastic Koopman operator associated with (\ref{eq:RandomRotation}), the functions $\varphi_n(x)=e^{inx}$ are eigenfunctions with eigenvalues $\lambda_i=\E(e^{in(\nu+\pi(\omega)})$, since
	\[\K\varphi_i(x)=\E(\varphi_i(T_\omega x))=\int_\Omega e^{in(x+\nu+\pi(\omega))}dP=e^{inx}\int_\Omega e^{in(\nu+\pi(\omega))}dP=\varphi_i(x)\lambda_i.\]
	
	We can compare these eigenvalues with the results obtained from our different DMD algorithms. We will set our system parameter to $\nu=\frac{1}{2}$ and draw $\pi(\omega_t)$ from the uniform distribution over $[-\frac{1}{2},\frac{1}{2}]$. In this case the eigenvalues are $\lambda_i=\frac{i-ie^{in}}{n}$. For the first test, we will compare Algorithms 1 and 3 using a set of observables with measurement noise. We will let our observable be 
	\begin{equation}
	\label{eq:RotObs1}
	\hat{\mathbf{f}}(t)=[\sin(x_t),...,\sin(5x_t),\cos(x_t),...,\cos(5x_t)]^T+\mathbf{m}(t),\end{equation} where $\mathbf{m}(t)\in[-0.5,0.5]^{10}$ is measurement noise drawn from the uniform distribution. Algorithm 1 is applied directly to the data from measurements of $\tilde{\mathbf{f}}$, and for Algorithm 3 we let $\tilde{\mathbf{g}}(t)=\tilde{\mathbf{f}}(t-1)$.
	
	For the second test, we let $f=\sin(x)+\sin(2x)+\sin(3x)$, and use time delays to generate $\hat{\mathbf{f}}$:
	\begin{equation}
	\label{eq:RotObs2}
	\hat{\mathbf{f}}(t)=\begin{bmatrix}
	f(x_t) & f(x_{t+1}) & \hdots & f(x_{t+d}).
	\end{bmatrix}^T.
	\end{equation}
	To perform Hankel DMD, we take five time delays ($d=5$ in (\ref{eq:RotObs2})) to generate $\tilde{\mathbf{f}}$, and use the data directly in Algorithm 1. However, if we try to perform Noise Resistant Hankel DMD using these observables, Algorithm 3 is poorly conditioned and and the eigenvalues are inaccurate. Instead, we use $24$ time delays of $\tilde{f}$ to generate $\tilde{\mathbf{f}}$ (setting $d=24$ in (\ref{eq:RotObs2}), and use Algorithm 4 (letting $\hat{\mathbf{g}}(t)=\hat{\mathbf{f}}(t-24)$) to truncate to the leading six singular values.
	Finally, we use Algorithm 4 again using only eight time delays to generate $\tilde{\mathbf{f}}$, but augment $\hat{\mathbf{g}}$ with extra observables to improve conditioning. We let $\hat{\mathbf{g}}$ to contain the observables $\hat{f},\hat{f}^2,$ and $\hat{f}^3$, as well as 42 time shifts of each of these functions:
	\[\hat{\mathbf{g}}=\begin{bmatrix}
	\hat{f}(t-42) & \hat{f}(t-42)^2 & \hat{f}(t-42)^3 & \hdots & \hat{f}(t) & \hat{f}(t)^2 & \hat{f}(t)^3
	\end{bmatrix}^T.\]

	\begin{figure}[htb]
		\label{fig:Rotation}
		\includegraphics[scale=0.165]{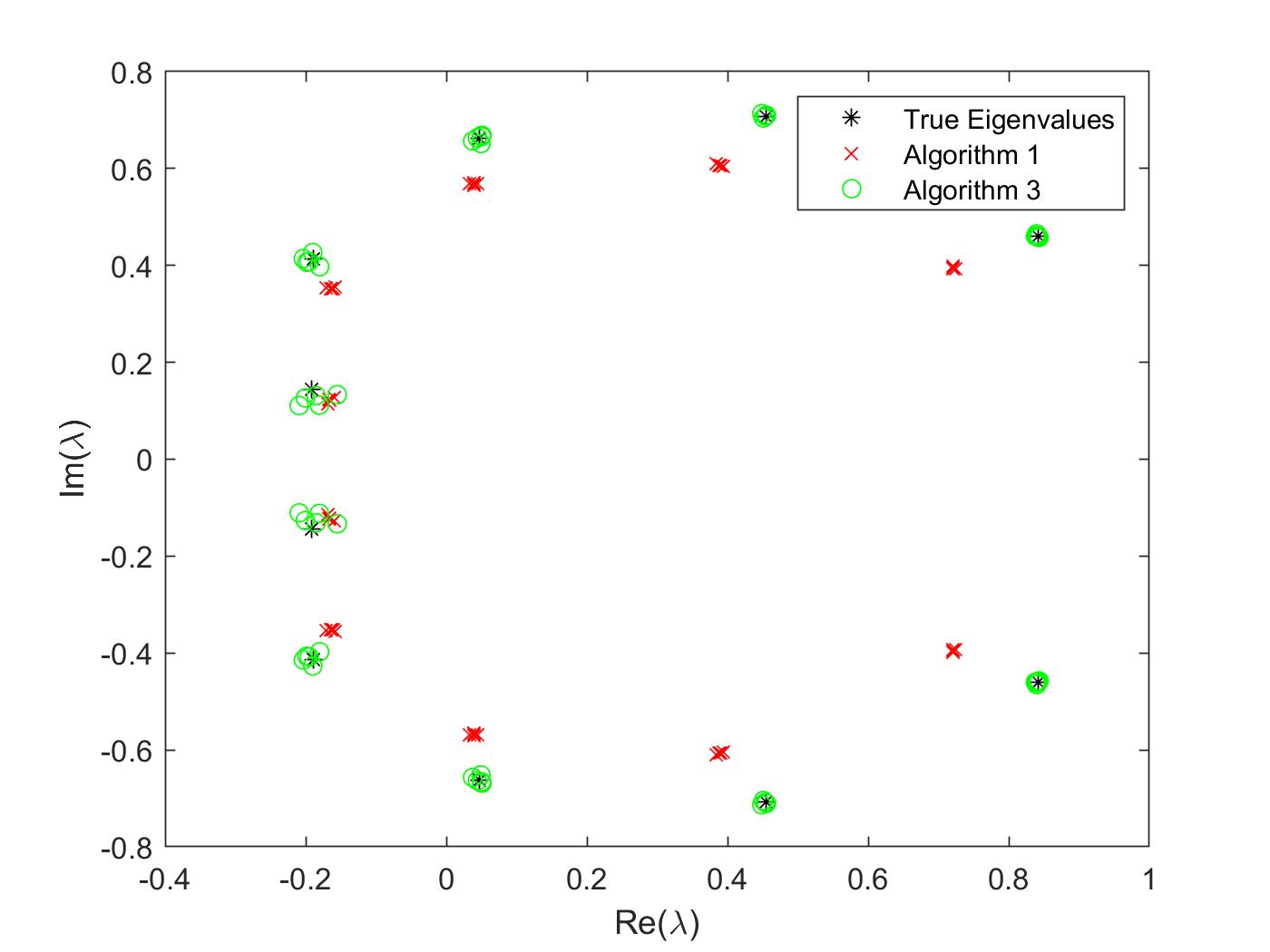}	
		\includegraphics[scale=0.165]{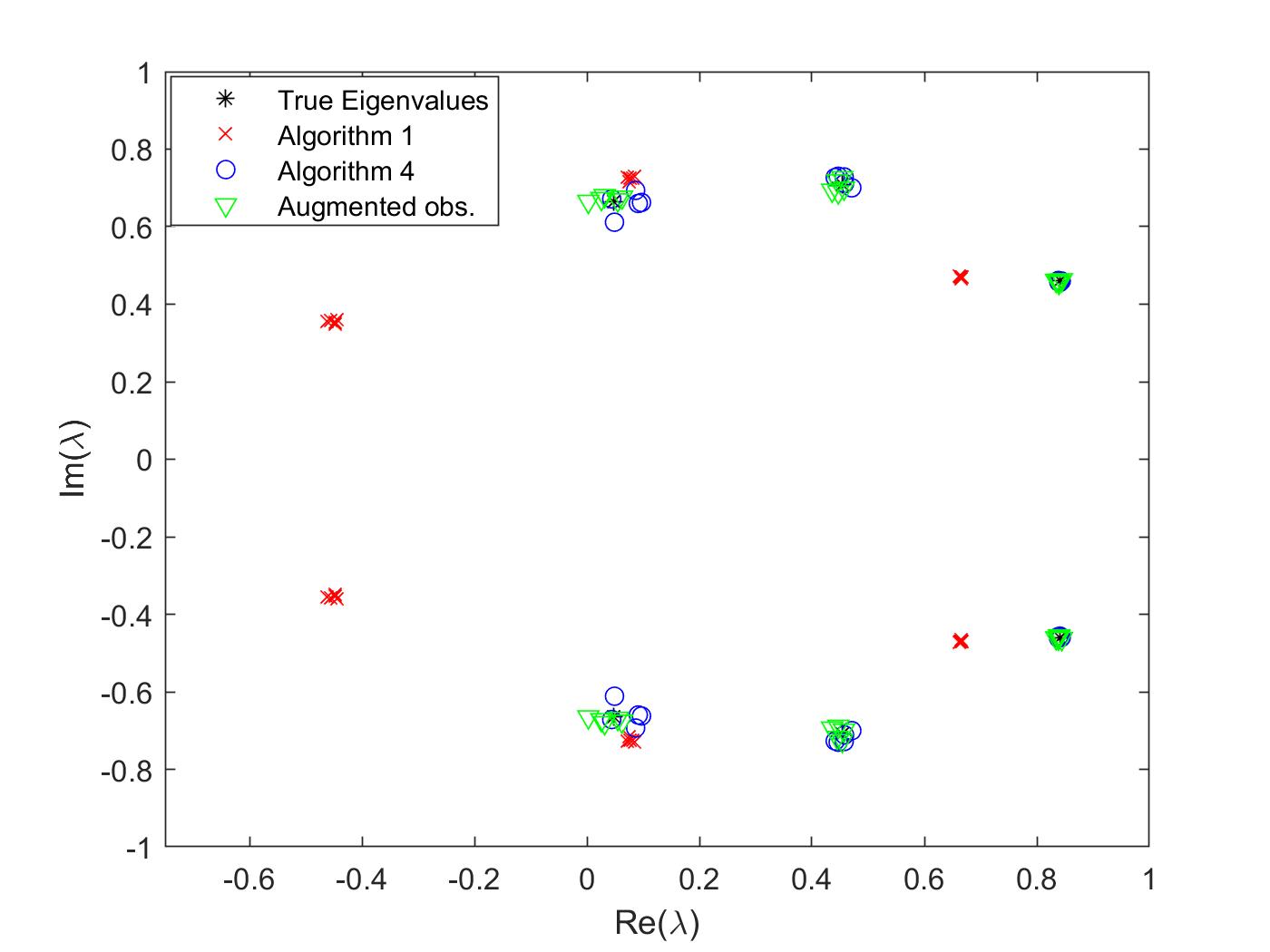}

		\caption{(Left) Outputs of Algorithm 1 and Algorithm 3 using (\ref{eq:RotObs1}) as observables on (\ref{eq:RandomRotation}) with 25 000 data points. Algorithm 1 shows a clear bias in the approximate eigenvalues while Algorithm 3 captures them accurately.\\
			(Right) DMD outputs from Algorithms 1 and 4 using (\ref{eq:RotObs2}) as observables on $(\ref{eq:RandomRotation})$ with 25 000 data points. Algorithm 4 is performed a second time after augmenting the dual observable to improve conditioning. Algorithm 1 shows a bias in the eigenvalues while Algorithm 4 gives an unbiased approximation of the eigenvalues in both cases. \\ Each algorithm is run five times on different sample trajectories.}
	\end{figure}
	
	As can be seen in Figure \ref{fig:Rotation}, Algorithm 1 fails to accurately approximate the eigenvalues of $\K$ in both tests. For the first test, Algorithm 3 gives accurate approximations to the eigenvalues of $\K$. Approximating the stochastic Koopman operator using the time delayed observables, (\ref{eq:RotObs2}) is more difficult because the conditioning of the matrix $G_0$ is very poor, which amplifies the errors in our time averages. However, including extra time delays and using Algorithm 4 to truncate to the leading singular values obtains accurate results. Further, the precision is increased when we augment $\tilde{\mathbf{g}}$ with extra observables.
	
	\subsection{Linear System with Additive Noise}
	Consider the linear system in $\R^4$:
	\begin{equation}
	\label{eq:Lineardet}
	\mathbf{x}(t+1)
	=\begin{bmatrix} 0.75 & 0.5 & 0.1 & 2 \\ 0 & 0.2 & 0.8 & 1 \\ 0 & -0.8 & 0.2 & 0.5 \\ 0 & 0 & 0 & -0.85 \end{bmatrix}
	\begin{bmatrix} x_1(t) \\ x_2(t)\\ x_3(t) \\ x_4(t)\end{bmatrix}=A\mathbf{x}(t).
	\end{equation}
	We can perturb (\ref{eq:Lineardet}) by perturbing the matrix $A$ with a random matrix $\delta$ and adding a random forcing term $b$. We obtain the random system
	\begin{equation}
	\label{eq:Linear}
	\mathbf{x}(t+1)
	=(A+\delta_t)\mathbf{x}(t)+b_t,
	\end{equation}
	where $b_t\in \R^4$ and $\delta_t\in \R^{4\times 4}$ are i.i.d. random variables. Let $(w_i,\lambda_i), i=1,...,4$ be the left eigenpairs of $A$. If $b_t$ and $\delta_t$ are assumed to have zero mean, $w_i^T\mathbf{x}$ is an eigenfunction of $\K$ with eigenvalue $\lambda_i$. For this example we will assume each component of $b_t$ and $\delta_t$ is drawn from randomly from a uniform distribution. The components of $b_t$ will be drawn from $[-0.5,0.5]$ while those of $\delta_t$ will be drawn from $[-0.25,0.25]$. As before, we will test Algorithms 1 and 3 using observables with measurement noise and time delayed observables. For the first test, we will use state observables with Gaussian measurement noise:
	\begin{equation}
	\label{eq:LinObs1}
	\hat{\mathbf{f}}(t)=\mathbf{x}(t)+\mathbf{m}(t)
	\end{equation}
	where each component of $\mathbf{m}(t)\in \R^4$ is drawn from the standard normal distribution. As before, will let $\hat{\mathbf{g}}(t)=\hat{\mathbf{f}}(t-1)$.
	
	For the second test, to generate the time delayed observables, we only use the first component of the state, $\hat{f}(t)=x_1(t)$, and use three time delays:
	\begin{equation}
	\label{eq:LinObs2}
	\hat{\mathbf{f}}(t)=\begin{bmatrix} \hat{f}(t) & \hat{f}(t+1) & \hat{f}(t+2) & \hat{f}(t+3)\end{bmatrix}.
	\end{equation}
	We will apply Algorithm 1 directly to this matrix, while for Algorithm 3 we let $\hat{\mathbf{g}}(t)=\hat{\mathbf{f}}(t-3)$. 
	
	\begin{figure}[htb]
		\label{fig:Linear}
		\includegraphics[scale=0.165]{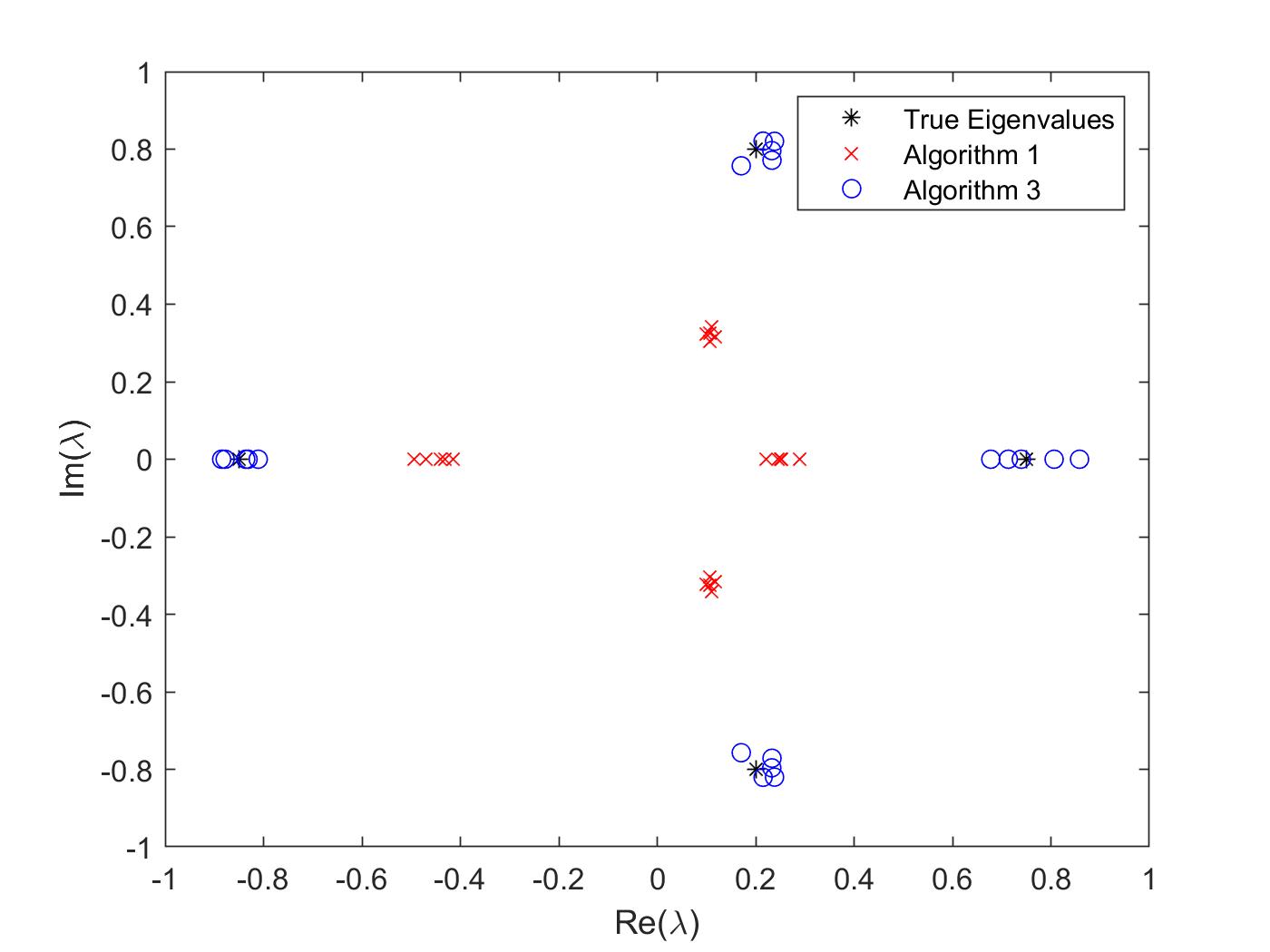}	
		\includegraphics[scale=0.165]{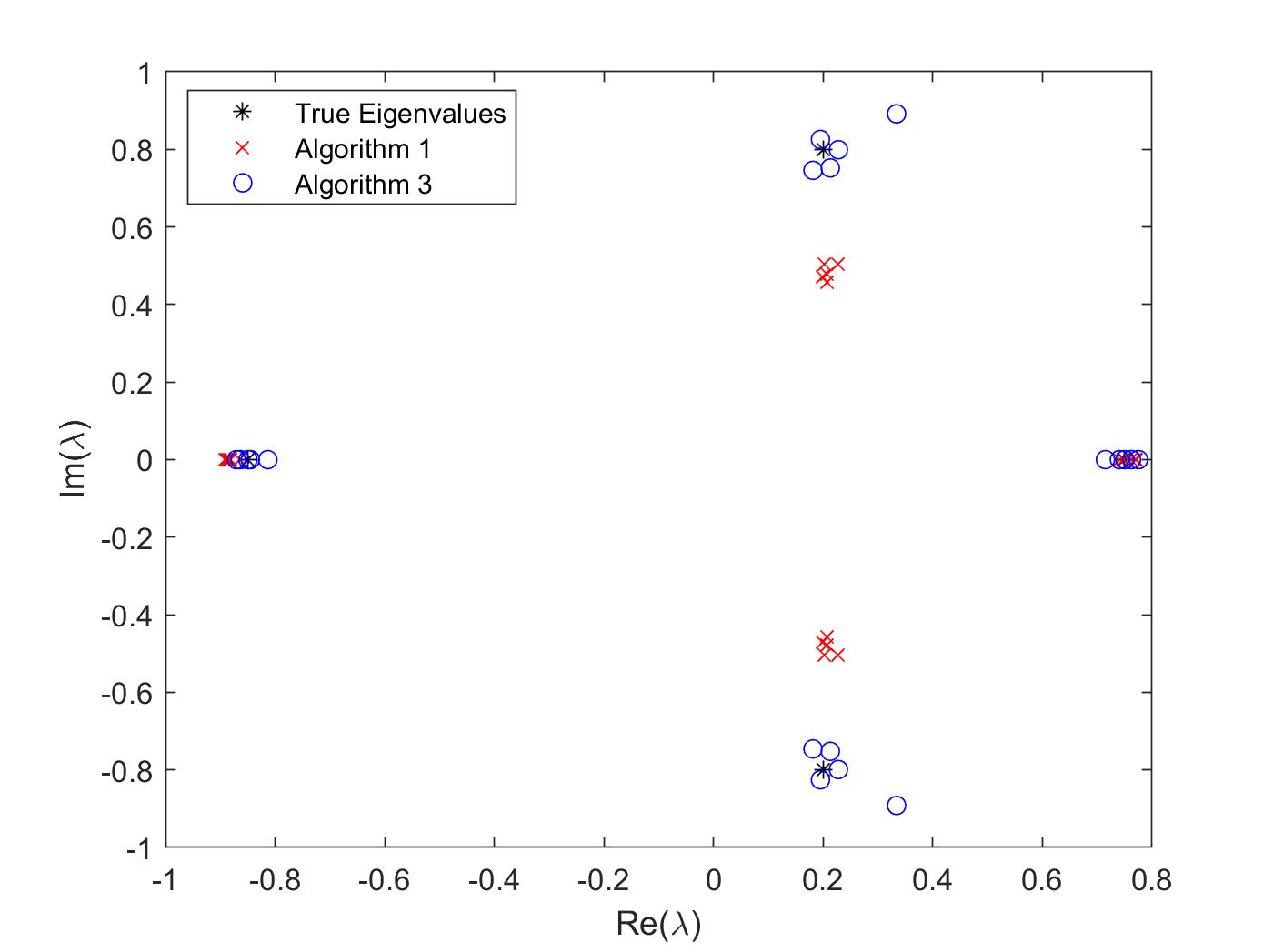}

		\caption{(Left) Outputs of Algorithm 1 and Algorithm 3 using state observables with measurement noise (\ref{eq:LinObs1}) on  5 000 data points from (\ref{eq:Linear}). \\(Right) Outputs of Algorithm 1 and Algorithm 3 using (\ref{eq:LinObs2}) as observables on $(\ref{eq:Linear})$ with 5 000 data points. For both cases, Algorithm 3 is unbiased in approximating the eigenvalues while algorithm 1 exhibits a clear bias. \\ Each algorithm is run five times on different sample trajectories.}
	\end{figure}
	
	Figure \ref{fig:Linear} shows that the eigenvalues generated by Algorithm 1 again fail to accurately approximate those of $\K$. However, for both sets of observables, Algorithm 3 estimates the eigenvalues of $\K$ well. Since we did not run into conditioning issues, we did not test the results using Algorithm 4 or an augmented dual observable.
	
	\subsection{Stuart Landau Equations}
	Consider the stochastic Stuart Landau equations defined by 
	\begin{align}
	\label{eq:Stuart1}
	dr &= (\delta r-r^3+\frac{\epsilon^2}{r})dt+\epsilon dW_r\\
	\label{eq:Stuart2}
	d\theta &= (\gamma-\beta r^2)dt+\frac{\epsilon}{r} dW_\theta,	
	\end{align}
	where $W_r$ and $W_\theta$ satisfy
	\begin{align*}
	dW_r&=\cos\theta\,dW_x + \sin\theta\,dW_y\\
	dW_\theta&=-\sin\theta\,dW_x+\cos\theta\,dW_y
	\end{align*}
	for independent Wiener processes $dW_x$ and $dW_y$. It was shown in \cite{Tantet} that for small $\epsilon$ and $\delta>0$, the (continuous time) stochastic Koopman eigenvalues are given by
	\begin{align*}
	\lambda_{l,n}=\begin{cases}
	-\frac{n^2\epsilon^2(1+\beta^2)}{2\delta}+in\omega_0+\mathcal{O}(\epsilon^4) & l=0\\
	-2l\delta+in\omega_0+\mathcal{O}(\epsilon^2) & l>0,
	\end{cases}
	\end{align*}
	where $\omega_0=\gamma-\beta\delta$.
	
	Let $\gamma=\beta=1$, $\delta=1/2$, and $\epsilon=0.05$ in (\ref{eq:Stuart1}) and (\ref{eq:Stuart2}). Define the observables
	\[f_k(r,\theta)=e^{ik(\theta-(\log(2r))}.\]
	First, we will let
	\begin{equation}
	\label{eq:StuartObs}
	\hat{\mathbf{f}}(t)=[f_1(x_t),f_{-1}(x_t),...,f_6(x_t),f_{-6}(x_t)]^T+\mathbf{m}_1(t)+i\mathbf{m}_2(t),
	\end{equation}
	where each component of $\mathbf{m}_1(t)$ and $\mathbf{m}_2(t)$ is drawn independently from a normal distribution with mean $0$ and variance $1/4$. In Algorithm 3, we let $\hat{\mathbf{g}}(t)=\hat{\mathbf{f}}(t-1)$. The (continuous time) eigenvalues generated by Algorithms 1 and 3 are shown from a simulation with 10,000 data points with a time step of 0.05 in Figure \ref{fig:Stuart}.
	
	To test Hankel DMD, we use the observable
	\[f=\sum_{k=1}^6(f_k+f_{-k}),\]
	and let $\tilde{\mathbf{f}}$ contain $f$ and $d$ time delays of $f$:
	\begin{equation}
	\label{eq:StuartObs2}
	\hat{\mathbf{f}}(t)=\begin{bmatrix}
	f(x_t) & f(x_{t+1}) & \hdots & f(x_{t+d})
	\end{bmatrix}.
	\end{equation}
	Due to the poor conditioning of Algorithms 1 and 3, the eigenvalues they generate are highly innaccurate, so we instead implement Algorithms 2 and 4. In each case, we let $d=399$ and truncate the SVD to the leading $12$ singular values. As usual, we let $\hat{\mathbf{g}}=\hat{\mathbf{f}}(t-d)$ in Algorithm 4. The results shown in Figure \ref{fig:Stuart} are from a simulation with 100,000 data points and a time step of 0.05.
	
	\begin{figure}[htb]
		\label{fig:Stuart}
		\includegraphics[scale=0.165]{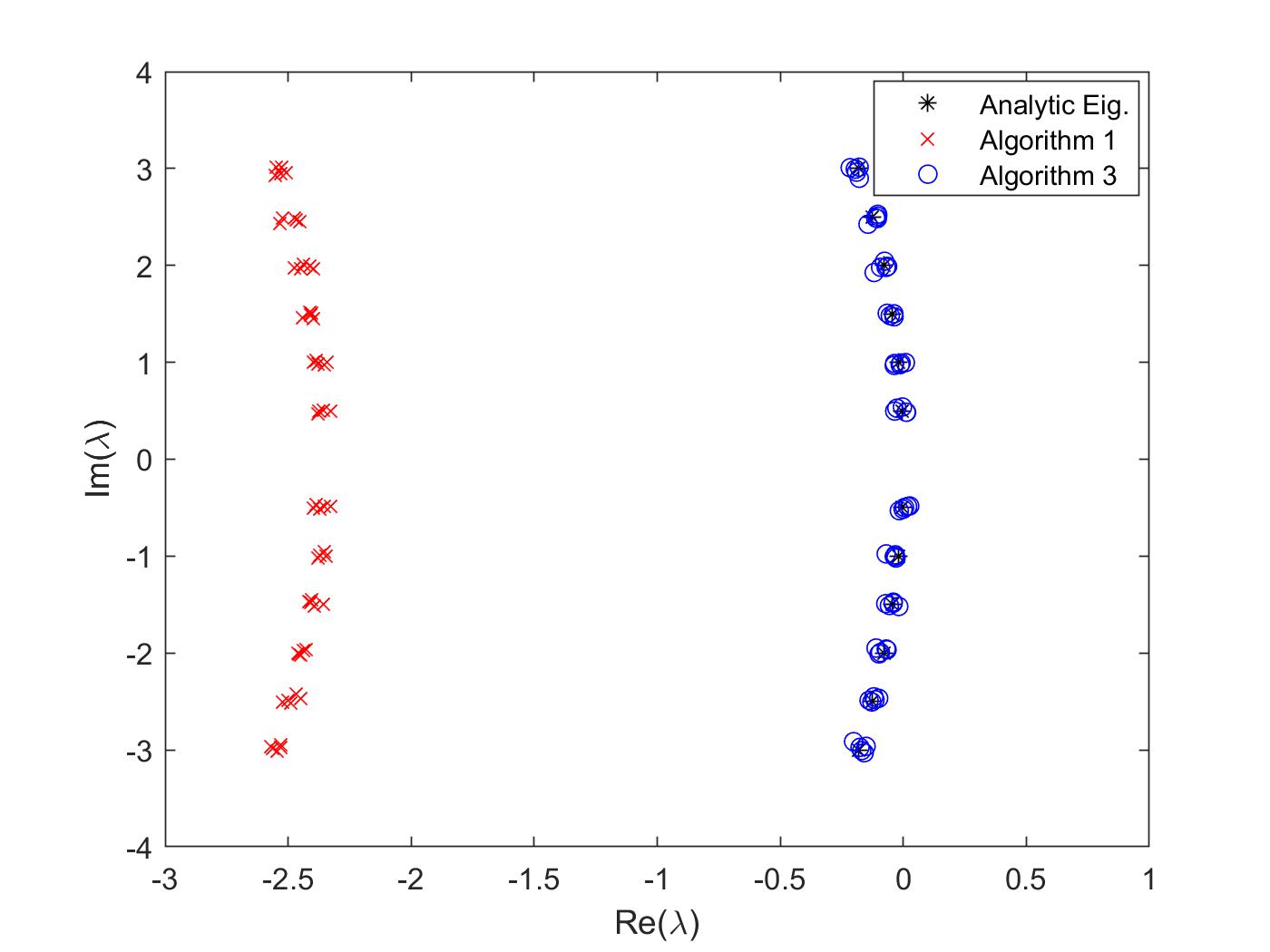}	
		\includegraphics[scale=0.165]{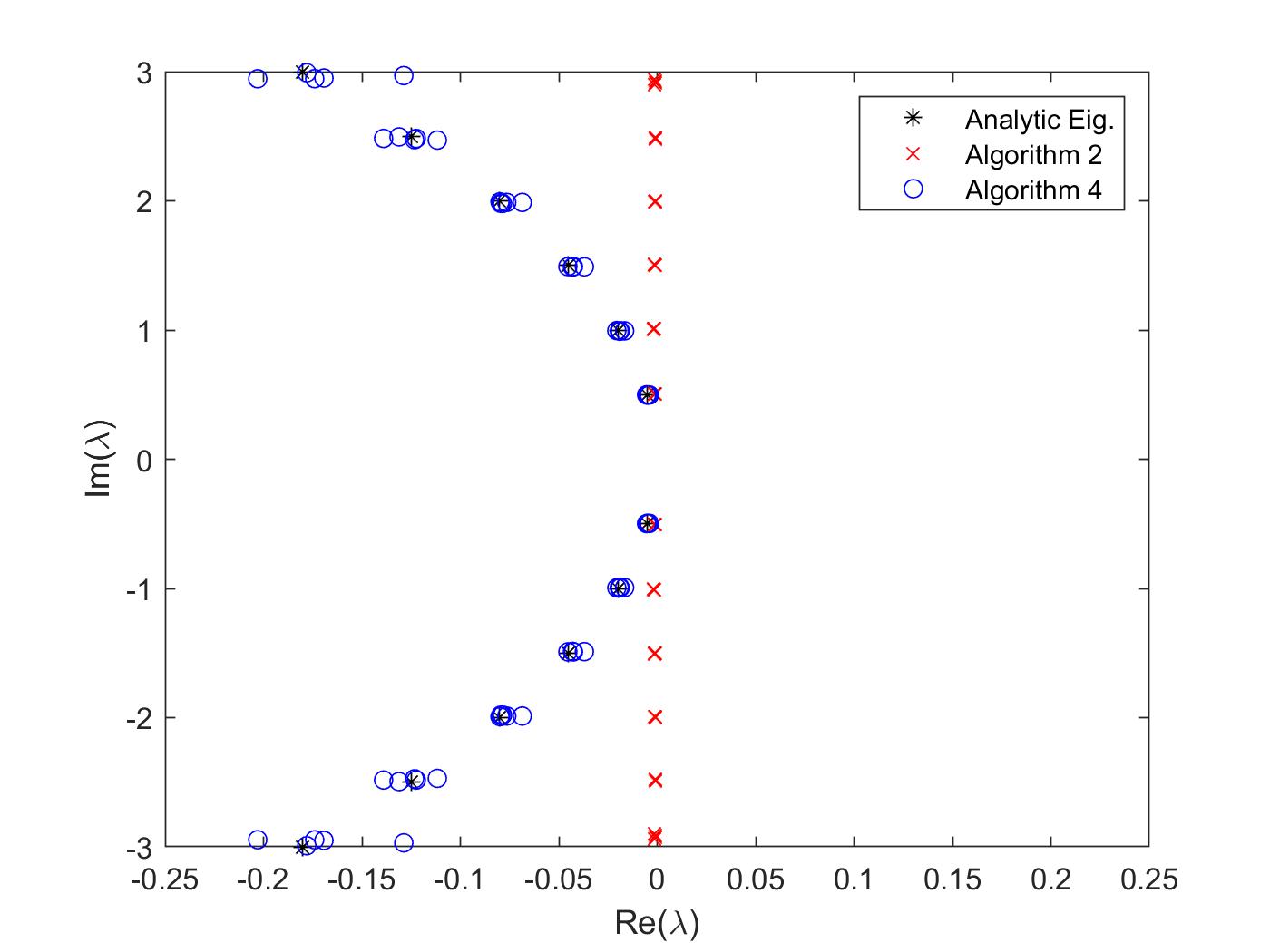}

		\caption{(Left) Outputs of Algorithm 1 and Algorithm 3 using observables with measurement noise (\ref{eq:StuartObs}). The data is taken over 20 000 data points from (\ref{eq:Stuart1}) and (\ref{eq:Stuart2}) with a time step of 0.05. The eigenvalues produced by Algorithm 1 are biased towards the left hand plane while Algorithm 1 captures them accurately.
			(Right) Outputs of Algorithm 2 and Algorithm 4 using (\ref{eq:StuartObs2}) as observables on (\ref{eq:Stuart1}) and (\ref{eq:Stuart2}). The Algorithms used 200 000 data points with a time step of 0.05. Algorithm 4 captures most of the eigenvalues without bias while Algorithm 2 biases all eigenvalues towards the imaginary axis. \\ Each algorithm is run five times on different sample trajectories.}
	\end{figure}

	As can be seen in Figure \ref{fig:Stuart}, Algorithm 1 exhibits a clear bias towards the left of the complex plane using observables with measurement noise, although it appears to accurately estimate the imaginary part of the eigenvalue. Algorithm 3, on the other hand, appears to give a mostly accurate spectrum. When using time delayed observables for Hankel DMD, Algorithms 1 and 3 were very poorly conditioned, and gave eigenvalues far outside the windows shown in Figure \ref{fig:Stuart}. When using Algorithms 2 and truncating to the 12 dominant singluar values, we again see that the imaginary parts of the eigenvalues seem to be captured, but the real parts are all biased to the right. Algorithm 4, however, again captures the correct spectrum, but with some error for the most dissipative eigenvalues.

	\section{Conclusions}
	
	In this paper we analyzed the convergence of DMD algorithms for random dynamical systems, culminating in the introduction of a new DMD algorithm that converges to the spectrum of the stochastic Koopman operator in the presence of both random dynamics and noisy observables. This allows us to avoid the bias in standard DMD algorithms that can come from ``overfitting" to the noise. We then specialized the algorithm to handle observables with i.i.d. measurement noise and time delayed observables and showed that measurements of a single set of observables was sufficient to generate an approximation of the stochastic Koopman operator. In particular, we demonstrated that a single trajectory of a single observable could be used to generate a Krylov subspace of the operator, which allows us to use DMD without needing to choose a basis of observables. 
	
	This algorithm provides a method for modeling complex systems where a deterministic model is unfeasible. This could be because a full state model would be to complex, observables of the full state are unavailable, or measurements come with uncertainty. A possible extension of this algorithm could adapt it to handle data from systems with control inputs, which could be used to develop control algorithms for random dynamical systems.\\
	
	\textbf{Acknowledgments:} 	This research was funded by the grants ARO-MURI W911NF-17-1-0306 and NSF EFRI C\# SoRo 1935327.

	\bibliographystyle{plain}
	\bibliography{Sources.bib}

\begin{thebibliography}{10}

\bibitem{Arbabi}
Hassan Arbabi and Igor Mezic.
\newblock Ergodic theory, dynamic mode decomposition, and computation of
  spectral properties of the koopman operator.
\newblock {\em SIAM Journal on Applied Dynamical Systems}, 16(4):2096--2126,
  2017.

\bibitem{Arnold1998}
Ludwig Arnold.
\newblock {\em Random Dynamical Systems}.
\newblock Springer-Verlag Berlin Heidelberg, 1998.

\bibitem{brunton2017chaos}
Steven~L Brunton, Bingni~W Brunton, Joshua~L Proctor, Eurika Kaiser, and
  J~Nathan Kutz.
\newblock Chaos as an intermittently forced linear system.
\newblock {\em Nature communications}, 8(1):1--9, 2017.

\bibitem{Nelly}
Nelida {\v{C}}rnjari{\'c}-{\v{Z}}ic, Senka Ma{\'c}e{\v{s}}i{\'c}, and Igor
  Mezi{\'c}.
\newblock Koopman operator spectrum for random dynamical systems.
\newblock {\em Journal of Nonlinear Science}, pages 1--50, 2019.

\bibitem{Dawson}
Scott~TM Dawson, Maziar~S Hemati, Matthew~O Williams, and Clarence~W Rowley.
\newblock Characterizing and correcting for the effect of sensor noise in the
  dynamic mode decomposition.
\newblock {\em Experiments in Fluids}, 57(3):42, 2016.

\bibitem{Gikhman}
I.I. Gikhman, A.V. Skorokhod, and S.~Kotz.
\newblock {\em The Theory of Stochastic Processes: I}.
\newblock Classics in Mathematics. Springer Berlin Heidelberg, 2004.

\bibitem{hemati2017biasing}
Maziar~S Hemati, Clarence~W Rowley, Eric~A Deem, and Louis~N Cattafesta.
\newblock De-biasing the dynamic mode decomposition for applied koopman
  spectral analysis of noisy datasets.
\newblock {\em Theoretical and Computational Fluid Dynamics}, 31(4):349--368,
  2017.

\bibitem{Kifer}
Yuri Kifer.
\newblock {\em Ergodic Theory of Random Transformations}.
\newblock Birkh{\"a}user Boston, Inc, 1986.

\bibitem{Koopman}
Bernard~O Koopman.
\newblock Hamiltonian systems and transformation in hilbert space.
\newblock {\em Proceedings of the national academy of sciences of the united
  states of america}, 17(5):315, 1931.

\bibitem{Kutz}
J~Nathan Kutz, Steven~L Brunton, Bingni~W Brunton, and Joshua~L Proctor.
\newblock {\em Dynamic Mode Decomposition: Data-Driven Modeling of Complex
  Systems}.
\newblock Other titles in applied mathematics. Society for Industrial and
  Applied Mathematics SIAM, 3600 Market Street, Floor 6, Philadelphia, PA
  19104, Philadelphia, Pennsylvania, 2016.

\bibitem{Mezic2000}
I~Mezic and Andrzej Banaszuk.
\newblock Comparison of systems with complex behavior: Spectral methods.
\newblock In {\em Proceedings of the 39th IEEE Conference on Decision and
  Control (Cat. No. 00CH37187)}, volume~2, pages 1224--1231. IEEE, 2000.

\bibitem{Mezic2005}
Igor Mezi{\'c}.
\newblock Spectral properties of dynamical systems, model reduction and
  decompositions.
\newblock {\em Nonlinear Dynamics}, 41(1-3):309--325, 2005.

\bibitem{Mezic2019}
Igor Mezi{\'c}.
\newblock Spectrum of the koopman operator, spectral expansions in functional
  spaces, and state-space geometry.
\newblock {\em Journal of Nonlinear Science}, pages 1--55, 2019.

\bibitem{Mezic2020}
Igor Mezic.
\newblock On numerical approximations of the koopman operator.
\newblock {\em arXiv preprint arXiv:2009.05883}, 2020.

\bibitem{PetersonErgodic}
Karl~Endel Petersen.
\newblock {\em Ergodic Theory}.
\newblock Cambridge studies in advanced mathematics ; 2. Cambridge University
  Press, Cambridge [Cambridgeshire] ;, 1983.

\bibitem{Rowley}
Clarence~W Rowley, Igor Mezi{\'c}, Shervin Bagheri, Philipp Schlatter, Dans
  Henningson, et~al.
\newblock Spectral analysis of nonlinear flows.
\newblock {\em Journal of fluid mechanics}, 641(1):115--127, 2009.

\bibitem{Schmid}
Peter Schmid and Joern Sesterhenn.
\newblock Dynamic mode decomposition of numerical and experimental data.
\newblock {\em APS}, 61:MR--007, 2008.

\bibitem{Kawahara}
Naoya Takeishi, Yoshinobu Kawahara, and Takehisa Yairi.
\newblock Subspace dynamic mode decomposition for stochastic koopman analysis.
\newblock {\em Physical Review E}, 96(3):033310, 2017.

\bibitem{Tantet}
A~Tantet, MD~Chekroun, HA~Dijkstra, and JD~Neelin.
\newblock Mixing spectrum in reduced phase spaces of stochastic differential
  equations.
\newblock {\em Part II: Stochastic Hopf Bifurcation. ArXiv e-prints}, 2017.

\bibitem{Williams2015data}
Matthew~O Williams, Ioannis~G Kevrekidis, and Clarence~W Rowley.
\newblock A data--driven approximation of the koopman operator: Extending
  dynamic mode decomposition.
\newblock {\em Journal of Nonlinear Science}, 25(6):1307--1346, 2015.

\end{thebibliography}

\end{document}